\documentclass[11pt,reqno]{amsart}
\usepackage{amsmath}
\usepackage{}
\usepackage{}
\usepackage{amsfonts}
\usepackage{a4wide}
\usepackage{amsmath,amsthm,amssymb,amscd}
\usepackage{latexsym}
\usepackage{hyperref}
\usepackage[numbers,sort&compress]{natbib}
\usepackage{hypernat}
\allowdisplaybreaks
\numberwithin{equation}{section}

\newtheorem{theorem}{Theorem}[section]
\newtheorem{proposition}[theorem]{Proposition}

\newtheorem{lemma}[theorem]{Lemma}

\theoremstyle{definition}

\newtheorem{remark}[theorem]{Remark}

\newcommand{\ds}{\displaystyle}

\newcommand{\hv}{{H_{V,\varepsilon}^s(\mathbb{R}^N)}}
\newcommand{\hs}{{H^s(\mathbb{R}^N)}}
\newcommand{\eq}{\eqref}
\def\r{\ref}
\def\ue{\|u\|_\varepsilon}

\def\c{\cite}
\def\sc{\mathcal C}
\def\G{\mathcal G}
\def\g{\mathfrak g}
\def\k{\kappa}

\newcommand{\Ds}{{\dot{H}^s(\mathbb{R}^N)}}

\def\P{\mathcal P_\varepsilon}
\def\I{I_\alpha*}
\def\x{\chi_\Lambda}
\def\i{I_\frac{\alpha}{2}*}
\def\j{J_\varepsilon}
\def\Fs{(-\Delta)^s}
\def\fs{(-\Delta)^{\frac{s}{2}}}
\def\u{|u|}
\def\d{\mathrm{d}}
\def\rn{\mathbb{R}^N}
\def\R{\mathbb R}
\def\N{\mathbb N}
\def\C{\mathbb C}

\def\S{C_c^\infty(\mathbb{R}^N)}
\def\wq{\infty}

\def\a{\alpha}
\def\be{\beta}
\def\ga{\gamma}

\def\la{\lambda}
\def\si{\sigma}

\def\var{\varphi}
\def\om{\omega}

\def\Ga{\Gamma}
\def\Om{\Omega}
\def\De{\Delta}

\newcommand{\dist}{{\rm dist}}

\newcommand{\va}{\varepsilon}

\usepackage{color}

\newcommand{\La}{\Lambda}

\begin{document}

\title[Fractional Choquard equations with decaying potentials]
{Semi-classical states for fractional Choquard equations  with  decaying potentials}

\thanks {The research was supported  by the Natural Science Foundation of China  (No. 12271196, 11931012).}
\author[ Y. Deng, S. Peng, X. Yang, ]{Yinbin Deng \textsuperscript{1}, Shuangjie Peng\textsuperscript{2} and Xian Yang \textsuperscript{3}}

{
\footnotetext[1]{School of Mathematics and Statistics \& Hubei Key Laboratory of Mathematical Sciences,
Central China Normal University,
Wuhan 430079, P. R. China. Email: ybdeng@ccnu.edu.cn.}
\footnotetext[2]{School of Mathematics and Statistics  \& Hubei Key Laboratory of Mathematical Sciences,
Central China Normal University,
Wuhan 430079, P. R. China. Email: sjpeng@ccnu.edu.cn}
\footnotetext[3]{ School of Mathematics and Statistics \& Hubei Key Laboratory of Mathematical Sciences,
Central China Normal University,
Wuhan 430079, P. R. China.   Email: yangxian@mails.ccnu.edu.cn.}
}

\begin{abstract}
   This paper deals with  the following  fractional Choquard equation
$$\varepsilon^{2s}(-\Delta)^su +Vu=\varepsilon^{-\alpha}(I_\alpha*|u|^p)|u|^{p-2}u\ \ \ \mathrm{in}\ \R^N,$$
where $\varepsilon>0$ is a small parameter,   $(-\Delta)^s$ is the  fractional Laplacian,  $N>2s$, $s\in(0,1)$, $\alpha\in\big((N-4s)_{+}, N\big)$, $p\in[2, \frac{N+\alpha}{N-2s})$, $I_\alpha$ is a Riesz potential, $V\in C\big(\R^N, [0, +\infty)\big)$ is an electric potential. Under some assumptions on the decay rate of $V$ and the corresponding range of $p$, we prove that the problem has a family of solutions $\{u_\varepsilon\}$ concentrating at a local minimum of $V$ as $\varepsilon\to 0$. Since the potential $V$ decays at infinity, we need to employ a type of penalized argument and  implement delicate analysis on the both nonlocal terms to establish regularity, positivity and asymptotic behaviour of $u_\varepsilon$, which is totally different from the local case. As a contrast, we also develop some nonexistence results, which imply that the assumptions on $V$ and $p$ for the existence of $u_\varepsilon$ are almost optimal. To prove our main results, a general
strong maximum principle and comparison function for the weak solutions of fractional Laplacian  equations are established.
The main methods in this paper are variational methods, penalized technique and some comparison principle developed  in this paper.


\end{abstract}

\maketitle
\section{Introduction}\label{s1}
In this paper, we study the following nonlinear fractional Choquard equation
\begin{equation}\label{eqs1.1}
\varepsilon^{2s}(-\Delta)^su +Vu=\varepsilon^{-\alpha}(I_\alpha*|u|^p)|u|^{p-2}u \ \ \  \mathrm{in}\ \R^N,
\end{equation}
 where~$\varepsilon>0$~is a parameter, $N>2s$, $s\in(0,1)$, $\alpha\in(0, N)$, $p\in[2, \frac{N+\alpha}{N-2s})$, $V\in C\big({\R}^N, [0, \infty)\big)$ is an external potential, $I_{\alpha}=A_{N,\a}|x|^{\alpha-N}$ is the Riesz potential with $A_{N,\a}=\frac{\Gamma(\frac{N-\a}{2})}{2^\a\pi^{N/2}\Ga(\frac{\a}{2})}$ (see \cite{rie}) and could be interpreted as the Green function of $(-\Delta)^{\frac{\alpha}{2}}$ in $\R^N$ satisfying the semigroup property $I_{\alpha+\beta}=I_\alpha*I_\beta$ for $\alpha,\beta>0$ such that $\alpha+\beta<N$, $(-\Delta)^s$ is the fractional Laplacian defined as
\begin{eqnarray*}
(-\Delta)^su(x)&:=&C(N,s)P.V.\int_{\mathbb{R}^N}\frac{u(x)-u(y)}{|x-y|^{N+2s}}~\d y\\
&=&C(N,s)\lim\limits_{r\to 0}\int_{\mathbb{R}^N\backslash B_r(x)}\frac{u(x)-u(y)}{|x-y|^{N+2s}}~\d y
\end{eqnarray*}
with $C(N,s)=\big(\int_{\rn}\frac{1-\cos(\zeta_1)}{|\zeta|^{N+2s}}\d\zeta\big)^{-1}$ (see \cite{ege2012}). In view of a path integral over the L\'{e}vy flights paths, the fractional Laplacian was introduced by Laskin (\cite{3}) to model fractional quantum mechanics.
When $s=\frac{1}{2}$, $N=3$ and $\a=2$, problem \eq{eqs1.1} is related to the following well-known boson stars equation (see \cite{es,on,bos,fro,len,lel})
\begin{equation}\label{q1r}
  i\partial_t\psi=\sqrt{-\De+m^2}\psi+(V(x)-E)\psi-(I_2*|\psi|^2)\psi,\quad \psi:[0,T)\times\R^3\to\C,
\end{equation}
which can effectively describe the dynamics and gravitational collapse of relativistic boson stars,
where $m\ge0$ is a mass
parameter and $\sqrt{-\De+m^2}$ is the kinetic energy operator defined via its symbol
$\sqrt{\xi^2 + m^2}$ in Fourier space. In the massless case ($m=0$), a standing wave $\psi(x,t):=e^{iEt}u(x)$ of \eq{q1r} leads to a solution $u$ of
\begin{align*}
 \sqrt{-\De}u+Vu=(I_2*|u|^2)u\quad \mathrm{in}\ \R^3.
\end{align*}

When $s=1$, equation \eqref{eqs1.1} boils down to the following classical Choquard equation:
\begin{equation}\label{la}
-\va^2\Delta u+Vu=\va^{-\a}(\I\u^p)\u^{p-2}u\ \ \ \mathrm{in}\ \rn,
\end{equation}
which was introduced by  Choquard in $1976$ in the modeling  of a one-component
plasma (\cite{lie}). The equation can also be derived
from the Einstein-Klein-Gordon and Einstein-Dirac system (\cite{giu}). Equation \eq{la} can be seen as a stationary nonlinear Schr\"odinger equation with an attractive long range interaction (represented by the nonlocal term) coupled with a repulsive short range interaction (represented by the local nonlinearity). While for the most of the relevant physical applications $p=2$, the case $p\neq 2$ may appear in several relativistic models of the density functional theory. When $V\in L^1_{\mathrm{loc}}(\R^N)$ is a non-constant electric potential, \eq{la} can model the physical  phenomenon in which particles are under the influence of an external electric field.

When $\varepsilon>0$ is a small parameter, which is typically related to the Planck constant, from the physical prospective \eq{eqs1.1} is particularly important, since its solutions  as $\varepsilon\to 0$ are called semi-classical bound states. Physically, it is expected that in the semi-classical limit $\varepsilon\to 0$ there should be a correspondence between solutions of the equation \eq{eqs1.1} and critical points of the potential $V$, which governs the classical dynamics.

For fixed  $\varepsilon>0$, for instance $\varepsilon=1$, problem \eqref{eqs1.1} becomes
\begin{equation}\label{fldd}
\Fs u+ V u=(\I\u^p)\u^{p-2}u\ \ \ \mathrm{in}\ \rn.
\end{equation}
In the case that $V(x)$ is a constant $\lambda>0$, $N\ge 3$ and $p\in (\frac{N+\a}{N},\frac{N+\alpha}{N-2s})$, it was verified in \cite{dss}  that problem \eq{fldd}
has a positive radial decreasing ground state $U_{\lambda}$. Moreover, if $p\ge2$, it holds  that $U_{\lambda}$ decays as follows:
\begin{equation}\label{sjg}
U_\la=\frac{C}{|x|^{N+2s}}+o(|x|^{-N-2s})\ \ \ \mathrm{as}\ |x|\to\wq
\end{equation}
for some $C>0$.

Noting that $I_\a*|u|^p\to |u|^p$ as $\a\to0$ for all $u\in C_0^\wq(\rn)$, we see that equation \eqref{eqs1.1} is formally associated to the following well-known fractional Schr\"{o}dinger equation:
\begin{equation}\label{flc}
\va^{2s}\Fs u+Vu=|u|^{2p-2}u\ \ \ \mathrm{in}\ \rn,
\end{equation}
which has been widely studied in recent years. For example, when $\varepsilon=1$ and $V\equiv \lambda>0$, by Fourier analysis and  extending  \eqref{flc}  into a local problem in $\R^{N+1}_{+}$(see \cite{24}), Frank et al.  in \cite{Frank.L-CPAM-2016} proved that  the  ground state of \eqref{flc} is unique up to translation. In \cite{paj2012}, it was proved that \eqref{flc} has a positive radial ground state when the nonlinear term is replaced by general nonlinear term. When $\varepsilon\to0$, it was shown in \cite{co2016}  and  \cite{xsc2019}  that $\eq{flc}$ has a family of solutions concentrating at a local minimum of $V$ in the nonvanishing case $\inf_{x\in\rn}V(x)>0$ and the vanishing case $\inf_{x\in\rn}V(x)|x|^{2s}<\wq$ respectively. For more results about \eq{flc}, we would like to refer the readers to  \cite{vab2017,jmj2014,ss2013,dmv,sil} and the references therein.

Inspired by the penalization method in \cite{mp1996} for \eq{flc} with $s=1$,
 Moroz et al. in \cite{Ms} introduced a novel penalized technique and obtained a family of single-peak solutions for \eqref{la} under various assumptions on the decay of $V$.

However, for the double nonlocal case, i.e., $s\in(0,1)$ and $\alpha\in (0,N)$,  there seems no result on the study of semi-classical solutions
 for \eq{eqs1.1} with vanishing potentials (particularly the potentials with compact support). If $V$ tends to zero at infinity,
the action functional corresponding to \eq{eqs1.1} is typically not well defined nor Fr\'echet differentiable
on $H_{V, \varepsilon}^s(\R^N)$ (which is defined later). Even in the local case $s=1$, this difficulty is not only technical. As was pointed out in \cite{mor1},  the local Choquard equations with fast decaying potentials indeed
may not have positive solutions or even positive super-solutions for certain
ranges of parameters. Hence the existence of semi-classical bound states to $\eq{eqs1.1}$ in the case $\liminf_{|x|\to\infty}V(x) = 0$ is an interesting but hard problem. In this paper, we will focus on the type of problems with   the potential $V$  decaying arbitrarily or even being compactly supported. It is worth pointing out that, compared with the local case $s=1$, the nonlocal effects from both $(-\Delta)^s\ (0<s<1)$ and the nonlocal nonlinear term will cause some new difficulties different from \cite{Ms,xsc2019}. For instance, the double nonlocal effects make it quite difficult to derive the uniform regular estimates and construct the penalized function and sup-solution.

In order to state our main results,  we first introduce some notations.

For~$0<s<1$~,~the usual fractional Sobolev space is defined as
$$
H^s(\R^N)=\big\{u\in L^2(\R^N):[u]_s<\infty\big\},
$$
endowed with the  norm
$\|u\|_{H^s(\rn)}=\big(\|u\|_{L^2(\rn)}^2+[u]_s^2\big)^{\frac{1}{2}}$,
where $[u]_s$ is defined as
$$
[u]_s^2:=\iint_{\R^{2N}}\frac{|u(x)-u(y)|^2}{|x-y|^{N+2s}}=\int_{\R^N}|(-\Delta)^{s/2}u|^2.
$$
For $N>2s$, we define the space $\dot H^s(\R^N)$ as
$$\dot H^s(\R^N)=\Big\{u\in L^{2_s^{\ast}}(\R^N):[u]_s^2<\infty\Big\},$$
which is the completion of $C_c^{\infty}(\R^N)$ under the norm $[u]_s$, where $2_s^{\ast}:=\frac{2N}{N-2s}$ is the fractional Sobolev critical  exponent.

Without loss of generality, hereafter,  we define $I_{\alpha}=\frac{1}{|x|^{N-\alpha}}$ and
\begin{equation*}
(-\Delta)^su(x):=2\lim\limits_{r\to 0}\int_{\mathbb{R}^N\backslash B_r(x)}\frac{u(x)-u(y)}{|x-y|^{N+2s}}~\d y.
\end{equation*}

Our study will rely on the following weighted Hilbert space
$$
H_{V, \varepsilon}^s(\R^N):=\Big\{u\in \dot{H}^s(\R^N):\int_{\R^N}V(x)|u|^2<\infty \Big\},$$
with the inner product
$$
\langle u,v\rangle_\varepsilon=\varepsilon^{2s}\iint_{\R^{2N}}\frac{\big(u(x)-u(y)\big)\big(v(x)-v(y)\big)}{|x-y|^{N+2s}}
+\int_{\R^N}
V(x)uv$$
and the corresponding norm
$$\|u\|_\varepsilon=\Big(\varepsilon^{2s}[u]_s^2+\int_{\R^N}V(x)|u|^2\Big)^{\frac{1}{2}}.$$


We assume  that $V$ satisfies the following assumption:

($\mathcal{V}$) $V\in C(\rn,[0,+\wq))$, and there exists a bounded open set~$\Lambda \subset \R^N$~such that
$$
0<V_0=\underset{x\in \Lambda}{\rm inf}V(x)<\underset{x\in \partial\Lambda}{\rm min} V(x).
$$
Moreover, we assume  without loss of generality that $0\in\Lambda$ and $\partial \La$ is smooth. From the assumption $(\mathcal{V})$, we choose a  smooth bounded  open set $U\subset\R^N$ such that $\Lambda\subset\subset U$ and $\inf_{x\in U\setminus\Lambda}V(x)>V_0$.~

We say that $u$ is a weak solution to equation $\eq{eqs1.1}$
if $u\in\hv$ satisfies
\begin{equation*}\label{445}
\langle u,\var\rangle_\va=\va^{-\a}\int_{\rn}(\I\u^p)\u^{p-2}u\var
\end{equation*}
for any $\var\in \hv$.

For convenience, hereafter, given $\Om\subset\rn$ and $\tau>0$, we denote $C^\tau(\Om)=C^{[\tau],\tau-[\tau]}(\Om)$
with $[\tau]$ denoting the largest integer no larger than $\tau$.

Now we state our main results.
\begin{theorem}\label{thm1.1}
Let $V$ satisfy ($\mathcal{V}$), $N>2s$, $\a\in \big((N-4s)_+,N\big)$, $p\in [2,\frac{N+\a}{N-2s})$ satisfying  one of the following two assumptions:

$(\mathrm{\mathcal{Q}1})$ 
$p> p_*:=1+\frac{\max\{s+\frac{\a}{2},\a\}}{N-2s}$;

$(\mathrm{\mathcal{Q}2})$  $p> p_\om:=1+\frac{\a+2s}{N+2s-\om}$ if $\inf_{x\in\rn}(1+|x|^{\omega})V(x)>0$ for some $\omega\in(0,2s]$.

\noindent
Then there exists an $\varepsilon_0>0$ such that for all $\varepsilon\in(0,\varepsilon_0)$, problem $(\ref{eqs1.1})$ admits a positive weak solution $u_{\varepsilon}\in C^\sigma_{\mathrm{loc}}(\rn)\cap L^\wq(\rn)$ with $\sigma\in (0,\min\{2s,1\})$, which owns the following two properties:

 i)  $u_{\varepsilon}$ has a global maximum point  $x_{\varepsilon}\in~\bar{\Lambda}$  such that
   $$\lim\limits_{\varepsilon\to 0}V(x_{\varepsilon})=V_0~$$
   and
   \begin{equation*}\label{yyy}
     u_{\varepsilon}(x)\leq\frac{C\varepsilon^{\gamma}}{\varepsilon^{\gamma}+|x-x_{\varepsilon}|^{\gamma}}
   \end{equation*}
 for a  positive constant $C$ independent of $\va$, where $\gamma>0$ is a positive constant close to $N-2s$ from below if $(\mathrm{\mathcal{Q}1})$ holds and close to $N+2s-\om$ from below if $(\mathrm{\mathcal{Q}2})$ holds;

  ii) $u_\va$ is a classical solution to \eq {eqs1.1} and $u_\va \in C_{\mathrm{loc}}^{2s+\vartheta}(\rn)$ for some $\vartheta\in (0,1)$ if $V\in C_{\mathrm{loc}}^\varrho(\rn)\cap L^\wq(\rn)$ for some $\varrho\in (0,1)$.
\end{theorem}

We also have the following nonexistence result, which implies  that the assumptions $(\mathrm{\mathcal{Q}1})$-$(\mathrm{\mathcal{Q}2})$ on $p$ and $V$ in Theorem \ref {thm1.1} are almost optimal.

\begin{theorem}\label{thm1.2'}
Let $N>2s$ and $V\in C(\rn,[0,+\wq))$. Then \eq{eqs1.1} has no nonnegative nontrivial continuous weak solutions  if
$p\in (1,1+\frac{s+\frac{\a}{2}}{N-2s})\cup[2, 1+\frac{\a}{N-2s})$ and $\limsup_{|x|\to\wq}(1+|x|^{2s})V(x)=0$.
\end{theorem}

\begin{remark}
 We do not need any extra assumptions on $V$ out of $\La$ in $(\mathrm{\mathcal{Q}1})$, which means that $V$ can decay arbitrarily even have compact support.
The restriction $p\ge2$ in Theorem \r{thm1.1} is crucially required since $u_\va^{p-2}$ will be unbounded if $p<2$.
Noting that $ (p_*,+\wq)\cap [2,\frac{N+\a}{N-2s})\subset(p_{2s},+\wq)\cap [2,\frac{N+\a}{N-2s})$ and $p_\om$ is decreasing on $\om\in (0,2s]$,
one can see from Theorem \r{thm1.1} that the restriction on $p$ is weaker when $V$ decays slower. Specially, when $\om<\min\{2s,N-\a\}$, the restriction on $p$ in $(\mathrm{\mathcal{Q}2})$ holds naturally since $p_\om<2$.

 The proof of our main results depends strongly on Proposition \r{tb}, which is a basis of applying comparison principle. We use a tremendous amount of delicate analysis to check Proposition \r{tb}.
\end{remark}


Let us now elaborate the main difficulties and novelties in our proof.

We will use the variational sketch to prove our results, hence it is natural to consider the following functional corresponding to $\eq{eqs1.1}$
\begin{equation}\label{zr}
E_\va(v):=\frac{1}{2}\ue^2-\frac{1}{2p\va^\a}\int_{\rn}|\i|u|^p|^2,\ v\in \hv,
\end{equation}
whose critical points are weak solutions of \eqref{eqs1.1}. However, $E_\va$ is not well defined when $V$ decays very fast. For example,  the function $\omega_\mu:=(1+|x|^2)^{-\frac{\mu}{2}}\in \hv$ but $\int_{\rn}|\i w_\mu^p|^2=+\wq$   for any $\mu\in(\frac{N-2s}{2},\frac{N+\a}{2p})$  if $V\le C(1+|x|^{2s})^{-1}$.
In addition, it is hard to verify directly the (P.S.) condition only under the local assumption $(\mathcal{V})$ on $V$.
Furthermore, due to the nonlocal effect of the Choquard term, if $V$ decays to 0 at infinity, it is very tricky to obtain a priori regular estimate desired for a weak solution $u$ of \eq{eqs1.1} because we neither know whether $u\in L^\wq(\rn)$ nor know  whether $I_\a*u^p\in L^\wq(\rn)$.
To overcome these difficulties, we employ a type of  penalized idea to modify the nonlinearity. We will introduce  the following penalized problem (see \eq{gp} and \eq{eqs3.2})
\begin{align}\label{Aeq1.7}
\quad\varepsilon^{2s} (-\Delta)^s u+V(x)u\nonumber
&=p\va^{-\a}\Big(\I\int_0^{u_+}\big(\chi_\La t_+^{p-1}+\chi_{\rn\setminus\La}\min\{t^{p-1}_+,\P(x)\}\big)\Big)\\
& \ \ \quad\quad\quad \times\Big(\chi_\La u_+^{p-1}+\chi_{\rn\setminus\La}\min\{u^{p-1}_+,\P(x)\}\Big).
\end{align}
 Under better pre-assumptions (see \eq{eqs2.1}, $(\mathcal{P}_1), (\mathcal{P}_2)$ in Section \r{sec2}) on the penalized function $\mathcal{P}_{\varepsilon}$, the functional corresponding to \eqref{Aeq1.7} is $C^1$ in $\hv$ and satisfies the (P.S.) condition. Hence the standard min-max procedure results in a critical point $u_{\varepsilon}$ which solves equation \eqref{Aeq1.7}. To prove that $u_\varepsilon$ is  indeed a solution to the original problem $(\ref{eqs1.1})$, a crucial step  is  to show that
\begin{equation}\label{Aeq1.8}
u^{p-1}_{\varepsilon}\le \mathcal{P}_{\varepsilon}\ \ \text{in}\ \R^N\backslash\Lambda,
\end{equation}
in which some new difficulties caused by the nonlocal term $(-\Delta)^s$($0<s<1$) and the nonlocal nonlinear term  will be involved.

Firstly, we need to prove the concentration of $u_{\varepsilon}$ (see Lemma \r{jz}). This step relies on the uniform regularity of $u_{\varepsilon}$.
However, under the double nonlocal effect of $(-\Delta)^s$ and the Choquard term, the regularity estimates here are non-trivial after the truncation of the nonlinear term (see \eqref{gp}). In \cite{dss}, using essentially the fact that week solutions of \eq{la} belongs to $L^2(\rn)$, some regularity results for solutions of \eq{la} were obtained. But in our case, the solutions $u_{\varepsilon}$ may not be $L^2$-integrable if especially $V$ is compactly supported. To overcome this difficulty, we first use directly the Moser iteration to get the uniform $L^\wq$-estimates (see Lemma \r{w}) and then apply a standard convolution argument (see \cite[Proposition 5]{re}) to get the uniform H\"{o}lder estimates.
Our proof is quite different from that of \cite{dss}, since the $L^2$-norm of $u_{\varepsilon}$ here is unknown for fast decay $V$.
 We emphasize here that the upper bound on the energy (see Lemma \r{lem4.5}) and the construction of the penalized function play a key role in the regularity estimates since we expect not only  the sufficient regularity
  estimates for fixed $\va>0$ but also the uniform  regularity estimates for all $\va\in (0,\va_0)$.

Secondly, the double nonlocal effects from the Choquard term and the operator $(-\Delta)^s$ make the construction of penalized function and sup-solution to the linearized equation (see \eqref{q1})
derived from the concentration of $u_{\varepsilon}$ more difficult than that in \cite{Ms,xsc2019}. By  large amounts of delicate nonlocal analysis, we find a sup-solution
$$w_\mu=\frac{1}{(1+|x|^2)^{\frac{\mu}{2}}},$$
where $\mu>0$ is a constant depending on different decay rates of $V$ (see the assumptions $(\mathcal{Q}_1)-(\mathcal{Q}_2)$ in Theorem \ref{thm1.1} above). We would like to emphasize that the sup-solutions above imply that the solutions $u_{\varepsilon}$ can decay fast than $|x|^{2s-N}$ or even $|x|^{-N}$ if $V$ decays slowly, which is quite different from \cite{xsc2019}. Moreover, the different behavior of $(-\De)^sw_\mu$ and $-\De w_\mu$, for instance
$(-\De)^sw_\mu\sim |x|^{-N-2s}$
 and
 $-\De w_\mu\sim |x|^{-\mu-2}$
 as $|x|\to\wq$ for any $\mu>N$, makes our proof quite different from that of \cite{Ms}.

Using the decay properties of $(-\De)^sw_\mu$ (see Proposition \r{tb}), we indeed provide a specific comparison function $w_\mu$ to derive decay estimates from above and below for solutions of general fractional equations. As an application, Proposition \r{tb} is used to the full in the proof of  Theorem \r{thm1.2'} by carrying out a skillful iteration procedure. We point out that it is interesting   that Proposition \r{tb} can  also be applied to the case $\inf_{\rn}V(x)>0$. For instance, for constant $\kappa>0$, instead of the comparison functions constructed by the Bessel Kernel (see \cite[Lemmas 4.2 and 4.3]{paj2012}),  function $w_{N+2s}(\la x)$ can be taken as a super-solution ($\la$ small) or a sub-solution ($\la$ large) to
$$(-\De)^s u+\kappa u=0,\quad |x|\ge R_\la$$
for some suitable $R_\la>0$.

The proof of Theorem \r{thm1.2'}  depends strongly on the positivity of solutions. To this end, we establish a general strong maximum principle for weak super-solutions (see \eq{wq9}).

It should be mentioned that the potential $V$  affects the decay properties of solutions. On one hand, assume that $c<(1+|x|^{2s})V(x)<C$ for $C,c>0$, then by Remark \r{rma}, $u_\va$ given by Theorem \r{thm1.1} satisfies
 $ u_\va\ge \frac{C_\va}{1+|x|^N}$ for some $C_\va>0$, and thereby
$$\limsup_{|x|\to\wq}u_\va(x)(1+|x|)^{N+2s}=+\wq.$$
 On the other hand,  we can check by the same way as that in \cite{dss}, that any nonnegative weak solution ${u}_\va$ to \eq{eqs1.1} must satisfy
$$\limsup_{|x|\to\wq}{u}_\va(x)(1+|x|)^{N+2s}<\wq,$$
for $p\in [2,\frac{N+\a}{N-2s})$ if $\inf_{x\in\rn}V(x)>0$.
 Hence, the solution $u_\va$ has different decay behavior at infinity  between the nonvanishing case ($\inf_{x\in\rn}V(x)>0$) and the vanishing case ($\lim_{|x|\to\wq}V(x)=0$).
In fact, we believe that solutions decay faster if $V$ decays slower (see the choice of $\ga$ in Theorem \r{thm1.1}) .

This paper will be organized  as follows: In Section \ref{sec2}, we modify the nonlinear term of \eqref{eqs1.1} and  get a new well-defined penalized functional whose critical point $u_{\varepsilon}$ can be obtained by min-max procedure in \cite{wm}. In Section \ref{sec3}, we give the essential energy estimates and regularity estimates of $u_{\varepsilon}$ and prove the concentration property  of $u_{\varepsilon}$. In Section \ref{sec4}, the concentration of $u_\va$ will be used to  linearize the penalized equation   for which we  construct a suitable super-solution and  the penalized function. We also  prove the decay estimates on $u_{\varepsilon}$ by comparison principle,  which shows that $u_{\varepsilon}$ solves indeed the origin problem \eqref{eqs1.1}. In Section \r{s6}, we present some nonexistence results and  verify Theorem \r{thm1.2'}.

 Throughout this paper, fixed constants are frequently denoted by $C>0$ or $c>0$, which may change from line to line if necessary, but are always independent of the variable under consideration. What's more, $\va\in (0,\va_0)$ and $\va_0$ can be taken smaller depending on the specific needs.

\vspace{0.3cm}

\section{The penalized problem}\label{sec2}

In this section, we introduce a  penalized functional which satisfies all the assumptions of Mountain Pass Theorem  by  truncating  the nonlinear term outside $\Lambda$, and  obtain a nontrivial Mountain-Pass solution $u_{\varepsilon}$ to the modified problem.

We first list the following inequalities which are essential in this paper.
\begin{proposition}\label{prop2.1}{\rm(\cite{ {lr2008}} Sharp\ fractional\ Hardy\ inequality)}
 Let~$N> 2s, s\in (0,1)$. Then for any $u\in~\dot{H}^s(\R^N)$, there exists  a constant $\mathcal{C}_{N,s}>0$ depending only on $N$ and $s$ such that
$$
\mathcal{C}_{N,s}\int_{\R^N}\frac{|u(x)|^2}{|x|^{2s}}\le [u]_s^2.
$$
\end{proposition}
\begin{proposition}\label{prop2.2}{\rm(\cite{{ege2012}} Fractional\ embedding\ theorem)}
Let $N>2s$, then the embeddings $\dot{H}^s(\R^N)\subset L^{2_s^*}(\R^N)$ and $H^s(\R^N)\subset L^q(\R^N)$ are continuous for any $q\in [2, 2_s^*]$.
Moreover, the following embeddings are compact
$$H^s(\rn)\subset L_{\mathrm{loc}}^q(\rn),\ \dot{H}^s(\rn)\subset L_{\mathrm{loc}}^q(\rn),\quad q\in[1,2_s^{\ast}).$$
\end{proposition}


\begin{proposition}\label{prop2.3}{\rm(Rescaled\ Sobolev\ inequality)}
Assume  $N>2s$ and $q\in [2, 2_s^{*}]$. Then for every
$u\in H^s_{V,\varepsilon}(\R^N)$, it holds
$$\int_{\Lambda}|u|^q\leq\frac{C}{\varepsilon^{N(\frac{q}{2}-1)}}\Big(\int_{\R^N}\varepsilon^{2s}|(-\Delta)^{s/2}u|^2+V|u|^2 \Big)^{\frac{q}{2}},$$
where $C>0$ depends only on $N$, $q$ and $V_0$.
\end{proposition}
\begin{proof}
Actually, by H\"{o}lder inequality, Young's inequality and Proposition \r{prop2.2}, we have
\begin{align*}
\begin{split}
\|u\|_{L^q(\Lambda)}&\leq\|u\|_{L^2(\Lambda)}^{\theta}\|u\|_{L^{2_s^*}(\Lambda)}^{1-\theta}\leq C\varepsilon^{-\beta}\|u\|_{L^2(\Lambda)}^{\theta}
\varepsilon^{\beta}[u]_s^{1-\theta}\\
&\leq C\theta\varepsilon^{-\frac{\beta}{\theta}}\|u\|_{L^2(\Lambda)}+C(1-\theta)\varepsilon^{\frac{\beta}{1-\theta}}[u]_s\\
&\leq \frac{C}{\varepsilon^{N(\frac{1}{2}-\frac{1}{q})}}\Big(\varepsilon^s[u]_s+\Big(\int_{\R^N}V|u|^2\Big)^{\frac{1}{2}}\Big),
\end{split}
\end{align*}
where $\frac{1}{q}=\frac{\theta}{2}+\frac{1-\theta}{2_s^*}$, $\beta=\theta N(\frac{1}{2}-\frac{1}{q})$ and $\inf_{\La}V=V_0>0$.
\end{proof}

\begin{proposition}\label{prop2.4}{\rm(\cite{lflm} Hardy-Littlewood-Sobolev\ inequality)} Let $N\in \mathbb{N}$, $\alpha\in (0, N)$ and $q\in (1, \frac{N}{\alpha})$. If $u\in L^q(\R^N)$, then $I_{\alpha}*u\in L^{\frac{Nq}{N-\alpha q}}$ and
$$
\Big(\int_{\mathbb{R}^{N}}\left|I_{\alpha} * u\right|^{\frac{N q}{N-\alpha q}}\Big)^{\frac{N-\a q}{Nq}} \leq C\Big(\int_{\mathbb{R}^{N}}|u|^{q}\Big)^{\frac{1}{q}}
,$$
where $C>0$ depends only on $\alpha$, $N$ and $q$.
\end{proposition}

\begin{proposition}\label{prop2.5}{\rm(\cite{sewg} Weighted Hardy-Littlewood-Sobolev\ inequality)} Let $N\in \mathbb{N}$, $\alpha\in (0, N)$. If
$u \in L^{2}\left(\mathbb{R}^{N},|x|^{\alpha} \mathrm{d} x\right)$, then $I_{\frac{\alpha}{2}} * u \in L^{2}\left(\mathbb{R}^{N}\right)$ and
$$
\int_{\mathbb{R}^{N}}|I_{\frac{\alpha}{2}} * u|^{2} \leq C_{\alpha} \int_{\mathbb{R}^{N}}|u(x)|^{2}|x|^{\alpha},
$$
where $C_{\alpha}=\frac{1}{2^{\alpha}}\Big(\frac{\Gamma\left(\frac{N-\alpha}{4}\right)}{\Gamma\left(\frac{N+\alpha}{4}\right)}\Big)^{2}$.
\end{proposition}

By the assumption $(\mathcal{V})$, we choose a family of nonnegative penalized functions $\mathcal{P}_{\varepsilon}\in L^\infty(\R^N)$ for $\va>0$ small in such a way that
\begin{equation}\label{eqs2.1}
\mathcal{P}_{\varepsilon}(x)=0~ \text{for}~ x\in~\Lambda\ \text{and}\  \lim\limits_{\varepsilon\to 0} \| \mathcal{P}_{\varepsilon}\|_{L^\infty (\R^N)}=0.
\end{equation}
The explicit construction of $\mathcal{P}_{\varepsilon}$ will be described later in Section \r{sec4}. Before that, we only need the following two embedding assumptions on $\mathcal{P}_{\varepsilon}$:

$\left(\mathcal{P}_{1}\right)$  the space $H_{V,\varepsilon}^s\left(\mathbb{R}^{N}\right)$ is compactly embedded into $L^{2}\left(\mathbb{R}^{N},\mathcal{P}_{\varepsilon}(x)^{2}|x|^{\alpha} \mathrm{d} x\right)$,

$\left(\mathcal{P}_{2}\right)$  there exists $\kappa\in (0,1/2)$ such that
\begin{equation}\label{a}
  \frac{p}{\varepsilon^{\alpha}} \int_{\mathbb{R}^{N}}\left|I_{\frac{\alpha}{2}} *\left(\mathcal{P}_{\varepsilon} u\right)\right|^{2} \leq
  \frac{pC_\a}{\varepsilon^{\alpha}} \int_{\mathbb{R}^{N}}|\mathcal{P}_{\varepsilon} u|^2|x|^\a
  \le\kappa \int_{\mathbb{R}^{N}} \varepsilon^{2s}|(-\Delta)^{\frac{s}{2}} u|^{2}+V(x)|u|^{2}
\end{equation}
for  $u \in \hv$, where $C_\alpha$ is given by Proposition \ref {prop2.5}.

Basing on the two assumptions above,  we define the penalized nonlinearity $g_{\varepsilon}:\R^N\times\R\to \R$ as
\begin{equation}\label{gp}
g_{\varepsilon}(x, t):=\chi_{\Lambda}(x) t_{+}^{p-1}+\chi_{\mathbb{R}^{N} \backslash \Lambda}(x) \min \big\{t_{+}^{p-1}, \mathcal{P}_{\varepsilon}(x)\big\},
\end{equation}
where $\chi_{\Omega}$ is the characteristic function corresponding to $\Omega\subset\rn$. Set $G_{\varepsilon}(x,t)=\int_0^tg_{\varepsilon}(x,r)~dr.$ One can check  that $g_{\varepsilon}(x, t) \leq t_{+}^{p-1}$ in $\rn\times\R$ and
\begin{align}\label{3y7}
&0 \leq G_{\varepsilon}(x, t)\le g_{\varepsilon}(x, t) t \leq t_{+}^{p} \chi_{\Lambda}+\mathcal{P}_{\varepsilon}(x) t_{+}\chi_{\Lambda^c}\quad \mathrm{in}\ \rn\times\R,\nonumber\\
&0 \leq G_{\varepsilon}(x, t) \leq \frac{1}{p} t_{+}^{p} \chi_{\Lambda}+\mathcal{P}_{\varepsilon}(x) t_{+}\chi_{\Lambda^c}\quad \mathrm{in}\ \rn\times\R,\\
&0 \leq pG_{\varepsilon}(x, t)=g_{\varepsilon}(x, t) t=t_+^p\quad \mathrm{in}\ \La\times\R.\nonumber
\end{align}
We consider the following penalized problem
\begin{equation}\label{eqs3.2}
\varepsilon^{2s} (-\Delta)^s u+V u=p \varepsilon^{-\alpha}\big(I_{\alpha} * G_{\varepsilon}(x,u)\big) g_{\varepsilon}(x,u) \quad \text { in } \mathbb{R}^{N},
\end{equation}
whose Euler-Lagrange functional $J_{\varepsilon}~:~H_{V,\varepsilon}^s(\R^N)\to \R$ is defined as
$$
J_{\varepsilon}(u)=\frac{1}{2}\|u\|_\varepsilon^2-\frac{p}{2 \varepsilon^{\alpha}} \int_{\mathbb{R}^{N}}|I_{\frac{\alpha}{2}} * G_{\varepsilon}(x,u)|^{2}.
$$
For  $u\in H_{V,\varepsilon}^s(\R^N)$, if $p\in [\frac{N+\alpha}{N}, \frac{N+\alpha}{N-2s}]$, by ($\mathcal{V}$), Propositions \ref{prop2.3} and \ref{prop2.4}, we have
\begin{equation}\label{b}
\frac{p}{\varepsilon^{\alpha}}\int_{\R^N}|I_\alpha*(\chi_\Lambda |u|^p)|^2\leq \frac{C}{\varepsilon^{\alpha}}\Big(\int_{\Lambda}|u|^{\frac{2Np}{N+\alpha}}\Big)^{\frac{N+\alpha}{N}}
\leq \frac{C}{\varepsilon^{(p-1)N}}\Big(\int_{\mathbb{R}^{N}} \varepsilon^{2s}|\fs u|^{2}+V|u|^{2}\Big)^p,
\end{equation}
where $\frac{2Np}{N+\alpha}\in[2, 2_s^*]$.

From \eq{a}, \eq{3y7}  and  \eq{b},  we conclude that
\begin{equation*}\label{c}
\frac{p}{2 \varepsilon^{\alpha}} \int_{\mathbb{R}^{N}}\left|I_{\frac{\alpha}{2}} * G_{\varepsilon}(x, u)\right|^{2}\leq C\|u\|_\varepsilon^2
+\frac{C}{\varepsilon^{(p-1)N}}\|u\|_\varepsilon^{2p},
\end{equation*}
which implies that $J_\varepsilon$ is well defined in $\hv$ if $\left(\mathcal{P}_{2}\right)$ holds.

Next, we prove that the functional $J_{\varepsilon}$ is $C^1$ in $\hv$.

\begin{lemma}\label{lem3.1}
If $p\in (\frac{N+\a}{N},\frac{N+\alpha}{N-2s})$ and $(\mathcal{P}_{1})$-$(\mathcal{P}_{2})$ hold, then
$J_{\varepsilon} \in C^{1}(H_{V,\varepsilon}^{s}\left(\mathbb{R}^{N})\right)$ and
$$
\left\langle J_{\varepsilon}^{\prime}(u), \varphi\right\rangle=\left\langle u, \varphi\right\rangle_\varepsilon-\frac{p}{\varepsilon^{\alpha}} \int_{\mathbb{R}^{N}}\big(I_{\alpha} * G_{\varepsilon}(x,u)\big) g_{\varepsilon}(x,u) \varphi,\,\,\forall\,u \in H_{V,\varepsilon}^{s}\left(\mathbb{R}^{N}\right),\,\varphi \in H_{V,\varepsilon}^{s}\left(\mathbb{R}^{N}\right).
$$
\end{lemma}

\begin{proof}
In fact, it suffices  to show that the nonlinear term
$$\mathcal{J}_{\varepsilon}:=\int_{\mathbb{R}^{N}}\left|I_{\frac{\alpha}{2}} * G_{\varepsilon}(x, u)\right|^{2}$$
is $C^1$ in $\hv$. Let $u_n\to u $ in $H_{V,\varepsilon}^s(\R^N)$. Noting that $\frac{2Np}{N+\alpha}<2_s^*$, from \eq{3y7}, $(\mathcal{P}_1)$, Propositions \ref{prop2.2}, \ref{prop2.4} and \ref{prop2.5}, we deduce that
\begin{align}\label{g}
&\int_{\R^N}|I_{\frac{\alpha}{2}}*\big(G(x,u_n)-G(x,u)\big)|^2\nonumber\\
\leq&2\int_{\R^N}|I_{\frac{\alpha}{2}}*\big(\chi_\Lambda(|u_n|^p-|u|^p)\big)|^2+2\int_{\R^N}|I_{\frac{\alpha}{2}}*(\mathcal{P}_\varepsilon|u_n-u|)|^2\nonumber\\
\leq&C\Big(\int_{\Lambda}(|u|_n^p-|u|^p)^{\frac{2 N}{N+\alpha}}\Big)^{\frac{N+\alpha}{N}}+C\int_{\mathbb{R}^{N}} |u_n-u|^{2}\mathcal{P}_{\varepsilon}^2|x|^{\alpha}\nonumber\\
=&o_n(1),
\end{align}
which yields that $\mathcal{J}_{\varepsilon}$ is continuous.

For any $\varphi \in \hv$ and $0<|t|<1$, by \eq{3y7}, it holds
\begin{align*}
&\quad\left||I_{\frac{\alpha}{2}} * G_{\varepsilon}(x, u+t\varphi)|^{2}-|I_{\frac{\alpha}{2}} * G_{\varepsilon}(x, u)|^{2}\right|/t\\
&\leq C\Big(\big|I_{\frac{\alpha}{2}}*\big((|u|^p+|\varphi|^p)\chi_{\Lambda}\big)\big|^2+
\big|I_{\frac{\alpha}{2}}*\big(\mathcal{P}_{\varepsilon}(|u|+|\varphi|)\big)\big|^2\Big)\in L^1(\R^N).
\end{align*}
Then by Dominated Convergence Theorem, we get
\begin{align*}
\langle\mathcal{J}_{\varepsilon}'(u),\varphi\rangle&=\lim\limits_{t\to 0}\int_{\R^N}\frac{|I_{\frac{\alpha}{2}} * G_{\varepsilon}(x, u+t\varphi)|^{2}-|I_{\frac{\alpha}{2}} * G_{\varepsilon}(x, u)|^{2}}{t}\\
        &=\int_{\R^N}\lim\limits_{t\to 0}\frac{|I_{\frac{\alpha}{2}} * G_{\varepsilon}(x, u+t\varphi)|^{2}-|I_{\frac{\alpha}{2}} * G_{\varepsilon}(x, u)|^{2}}{t}\\
        &=2\int_{\R^N}\Big(I_{\frac{\alpha}{2}} * G_{\varepsilon}(x, u)\Big)\Big(I_{\frac{\alpha}{2}} * \big(g_{\varepsilon}(x, u)\varphi\big)\Big),\\
        &=2\int_{\R^N}\big(I_\alpha * G_{\varepsilon}(x, u)\big)g_{\varepsilon}(x, u)\varphi,
\end{align*}
which indicates  the existence of Gateaux derivative.

For the continuity of $\mathcal{J}_{\varepsilon}'$, we observe that
\begin{align*}
\langle\mathcal{J}_{\varepsilon}'(u_n)-\mathcal{J}_{\varepsilon}'(u),\varphi\rangle
=&2\int_{\R^N}\big(I_{\frac{\alpha}{2}} * G_{\varepsilon}(x, u_n)\big)\Big(I_{\frac{\alpha}{2}} *\Big(\big(g_{\varepsilon}(x, u_n)-g_{\varepsilon}(x, u)\big)\varphi\Big)\Big)\\
&+2\int_{\R^N}\Big(I_{\frac{\alpha}{2}} * \big(G_{\varepsilon}(x, u_n)-G_{\varepsilon}(x, u)\big)\Big)\Big(I_{\frac{\alpha}{2}} *\big(g_{\varepsilon}(x, u)\varphi\big)\Big).
\end{align*}
Then, by H\"{o}lder inequality and calculations similar to \eqref{g}, we deduce that
\begin{align*}
&\big|\langle\mathcal{J}_{\varepsilon}'(u_n)-\mathcal{J}_{\varepsilon}'(u),\varphi\rangle\big|
= o_n(1)\|\varphi\|_\va.
\end{align*}
Hence  $\mathcal{J}_{\varepsilon}'(u)$ is continuous and the proof is completed.
\end{proof}



Furthermore, 
we deduce that $J_{\varepsilon}$ satisfies the (P.S.) condition.
\begin{lemma}\label{lem3.3}
If $p\in (\frac{N+\a}{N},\frac{N+\alpha}{N-2s})$ and $\left(\mathcal{P}_1\right)$-$\left(\mathcal{P}_2\right)$ hold, then $J_{\varepsilon}$ satisfies the (P.S.) condition.
\end{lemma}

\begin{proof}
By Lemma \r{lem3.1}, $J_{\varepsilon} \in C^{1}(H_{V,\varepsilon}^{s}\left(\mathbb{R}^{N})\right)$.
Let $\{u_n\}\subset \hv$ satisfy $J_{\varepsilon}(u_n)\leq c$ and $J_{\varepsilon}'(u_n)\to 0.$
We claim that $\{u_n\}$  is bounded in $\hv$. Indeed, by \eq{3y7}, we have
\begin{align}\label{d}
\begin{split}
J_{\varepsilon}(u_n)-\frac{1}{2}\langle J_{\varepsilon}'(u_n),u_n\rangle
=&\frac{p}{2\va^\a}\int_{\R^N}\big(\I G(x,u_n)\big)\big(g_\va(u_n)u_n-G_\va(x,u_n)\big)\\[3mm]
\ge&\frac{p-1}{2p\va^\a}\int_{\Lambda}\big(\I (\x u_{n+}^p)\big)u_{n+}^p.
\end{split}
\end{align}
On the other hand, in view of  \eq{3y7}, Young's inequality and \eq{a}, we see that
\begin{align}\label{e}
\frac{1}{2}\|u_n\|_\va^2=&\frac{p}{2\va^\a}\int_{\rn}|\i G(x,u_n)|^2\d x+\j(u_n)\nonumber\\
\le&\frac{p}{2\va^\a}\int_{\rn}\big|\frac{1}{p}\i (\chi_\La u_{n+}^p)+\i (\mathcal{P}_\va |u_n|)\big|^2\d x+\j(u_n)\nonumber\\
\le&\kappa\|u_n\|_{\varepsilon}^{2}+\frac{1}{p\va^\a}\int_{\mathbb{R}^{N}}\big|I_{\frac{\alpha}{2}} *\left(\chi_{\Lambda} u_{n+}^{p}\right)\big|^{2}+\j(u_n),
\end{align}
Then it holds from $\kappa<1/2$ and \eq{d}--\eq{e} that
\begin{align}\label{f}
\|u_n\|_\va^2\le C'_1\j(u_n)+C'_2|\langle\j'(u_n),u_n\rangle|,
\end{align}
where $C'_1, C'_2>0$ are constants independent of $\va$. Then $\|u_n\|_\va\le C$.
Up to a subsequence, we have $u_n\rightharpoonup u$ in $\hv$.

By the same proof as \eqref{g}, we have
\begin{align}\label{d1}
\begin{split}
\int_{\mathbb{R}^{N}}\big(I_\a * G_{\varepsilon}(x,u_{n})\big)g_{\varepsilon}(x,u_{n}) u_{n}\to \int_{\mathbb{R}^{N}}\big(I_\a * G_{\varepsilon}(x,u)\big)g_{\varepsilon}(x,u) u
\end{split}
\end{align}
and
\begin{align*}
\int_{\mathbb{R}^{N}}\big(I_\a * G_{\varepsilon}(x,u_{n})\big)g_{\varepsilon}(x,u_{n}) u\to \int_{\mathbb{R}^{N}}\big(I_\a * G_{\varepsilon}(x,u)\big)g_{\varepsilon}(x,u) u.
\end{align*}
It follows from $u_n\rightharpoonup u$ in $\hv$ that
\begin{align}\label{ddh}
 0=\lim_{n\to\wq}\langle\j'(u_n),u\rangle=\ue^2-\frac{p}{\va^\a}\int_{\mathbb{R}^{N}}\big(I_\a * G_{\varepsilon}(x,u)\big)g_{\varepsilon}(x,u) u.
\end{align}
Combining $\eq{d1}$ with \eq{ddh}, we get
\begin{align*}\label{d3}
\lim_{n\to\wq}\|u_n-u\|_\va^2=&\lim_{n\to\wq}(\|u_n\|_\va^2-\|u\|_\va^2)\\
=&\lim_{n\to\wq}\frac{p}{\va^\a}\Big(\int_{\mathbb{R}^{N}}\big(I_{\alpha} * G_{\varepsilon}(x, u_{n})\big) g_{\varepsilon}(x, u_{n}) u_{n}-\int_{\mathbb{R}^{N}}\big(I_{\alpha} * G_{\varepsilon}(x, u)\big) g_{\varepsilon}(x, u) u\Big)\\
&+\lim_{n\to\wq}\langle\j'(u_n),u_n\rangle=0,
\end{align*}
which completes the proof.
\end{proof}


Finally, it is easy to check that $J_{\varepsilon}$ owns the Mountain Pass Geometry, so by  Lemma \ref {lem3.1} and Lemma \ref {lem3.3}, we can find a critical point for $J_{\varepsilon}$ via min-max theorem (\cite{wm}).

Define the Mountain-Pass value $c_{\va}$  as
\begin{equation}\label{Ade2.11}
  c_{\va}:=\inf_{\gamma\in \Gamma_{\va}}\max_{t\in[0,1]}J_{\varepsilon}(\gamma(t)),
\end{equation}
where
$$\Gamma_{\va}:=\Big\{\gamma\in C\big([0,1],\hv\big)\mid\gamma(0)=0,\ J_{\varepsilon}\big(\gamma(1)\big)<0\Big\}.$$

We have the following lemma immediately.
\begin{lemma}\label{lem3.6}Let $p\in (\frac{N+\a}{N},\frac{N+\alpha}{N-2s})$ and $\left(\mathcal{P}_1\right)$-$\left(\mathcal{P}_2\right)$ hold.
Then $c_{\varepsilon}$ can be achieved by a $u_{\va}\in \hv\setminus\{0\}$, which is a nonnegative  weak solution of the penalized equation $(\ref{eqs3.2})$.
\end{lemma}

\begin{proof}
The existence is trivial by  Lemmas \ref{lem3.1}, \ref{lem3.3} and the min-max procedure in \cite{wm}.

Letting $u_{\va,-}$ be a test function in
$\eq{eqs3.2}$, we obtain
\begin{equation}\label{d31}
\begin{aligned}
&\va^{2s}\iint_{\mathbb{R}^{2 N}} \frac{|u_{\va,-}(x)-u_{\va,-}(y)|^{2}}{|x-y|^{N+2 s}}+\int_{\mathbb{R}^{N}} V|u_{\va,-}|^{2}\\
\le&\va^{2s}\iint_{\mathbb{R}^{2 N}} \frac{\big(u_{\va,-}(x)-u_{\va,-}(y)\big)\big(u_{\va,+}(x)-u_{\va,+}(y)\big)}{|x-y|^{N+2 s}}\le0,
\end{aligned}
\end{equation}
which leads to $u_{\va,-}=0$ and thereby $u_{\varepsilon}$ is nonnegative.
\end{proof}

To expect the positivity of $u_\va$, we give the following strong maximum principle.
\begin{lemma}\label{w61}
Let  $c(x)\in L^\wq_{\mathrm{loc}}(\rn)$ and $u\in \dot{H}^s(\rn)$ be a weak supersolution to
\begin{align}\label{wq9}
  (-\De)^su+c(x)u=0,\quad x\in\rn.
\end{align}
 If $u\in C(\rn)$ and $u\ge0$ in $\rn$, then either $u\equiv0$ in $\rn$ or $u>0$ in $\rn$.
\end{lemma}
\begin{proof}
Suppose by contradiction that there exist $x_0,y_0\in\rn$ such that $u(x_0)=0$ and $u(y_0)>0$.
Denote
$$r:=\frac{|x_0-y_0|}{2},\ R:=2\max\{|x_0|,|y_0|\},\ \sigma:=\|c(x)\|_{L^\wq(B_r(x_0))},\ M:=\max_{B_R(0)}u(x).$$
Clearly, $B_r(x_0)\subset B_R(0)$, $y_0\in B_R(0)\backslash B_r(x_0)$ and $u$ weakly satisfies
\begin{align}\label{qgs}
  (-\De)^su+\si u\ge(\si-c(x))u\ge0,\quad x\in B_r(x_0).
\end{align}
Define $\bar{u}=\min\{M,u(x)\}$. We see that $\bar{u}\in C(\rn)\cap L^\wq(\rn)$, $0\le \bar{u}\le u(x)$ in $\rn$ and $\bar{u}=u(x)$ in $B_R(0)$.
Moreover, since $|\bar{u}(x)-\bar{u}(y)|\le|u(x)-u(y)|$, we deduce that $\bar{u}\in \dot{H}^s(\rn)$.

We claim that the following problem
\begin{equation}\label{w18}\left\{
  \begin{aligned}
    (-\De)^sv+\si v=&0,\quad x\in B_r(x_0),\\
    v=&\bar{u},\quad x\in \rn\backslash B_r(x_0)
  \end{aligned}\right.
\end{equation}
has  a weak solution $v\in \dot{H}^s(\rn)$.

Indeed, define the following Hilbert space
$$ \mathcal{H}_0^s(B_r(x_0)):=\big\{\phi\in H^s(\rn): \phi\equiv0\ \mathrm{on}\ \rn\backslash B_r(x_0)\big\}.$$
Since $(-\De)^s\bar{u}+\si \bar{u}\in \big(\mathcal{H}_0^s(B_r(x_0))\big)^{-1}$ in the sense of
\begin{align*}
 \big \langle (-\De)^s\bar{u}+\si \bar{u}, \phi\big\rangle:=\int_{\rn}(-\De)^{s/2}\bar{u}(-\De)^{s/2}\phi+ \int_{B_r(x_0)}\si \bar{u}\phi,\quad \phi\in \mathcal{H}_0^s(B_r(x_0)),
\end{align*}
 it follows from Riesz representation theorem that there exists $w\in \mathcal{H}_0^s(B_r(x_0))$ satisfying weakly
\begin{equation*}\left\{
  \begin{aligned}
    (-\De)^sw+\si w=&-(-\De)^s\bar{u}-\si \bar{u},\quad& x\in B_r(x_0),\\
    w=&0,\quad& x\in \rn\backslash B_r(x_0).
  \end{aligned}\right.
\end{equation*}
Consequently,  $v=\bar{u}+w$ solves \eq{w18} in the weak sense.

Let $v\in \dot{H}^s(\rn)$ be  a weak solution of \eq{w18}, using \eq{qgs}-\eq{w18} and comparison principle  we deduce
\begin{align}\label{xye}
  v(x)\le u(x),\quad x\in B_r(x_0).
\end{align}
Since $\bar{u}=u$ in $B_r(x_0)$, it follows that $v(x)\le \bar{u}$ in $\rn$.
On the other hand, taking $v_-$ as a test function in  \eq{w18}, we have $v\ge0$ in $\rn$. As a result, $0\le v\le \bar{u}$ and $v\in L^{\wq}(\rn)$.
 By the regularity theory in \cite[Proposition 5]{re} and \cite[Theorem 12.2.5]{clp}, there holds
$v\in C_{\mathrm{loc}}^{2s+\ga}(B_r(x_0))$ for some $\ga>0$, which implies $v$ is a classical solution to \eq{w18}.
If $v(x_0)=0$, then we have
\begin{equation*}
  C(N,s)P.V.\int_{\rn}\frac{0-v(y)}{|x_0-y|^{N+2s}}=(-\De)^sv(x_0)+\si v(x_0)=0,
\end{equation*}
which and $v(y)\ge0$ implies that $v\equiv0$ in $\rn$. This contradicts to $v(y_0)=\bar{u}(y_0)=u(y_0)>0$. Therefore, $v(x_0)>0$ and thereby $u(x_0)\ge v(x_0)>0$,
 which contradicts to $u(x_0)=0$.
\end{proof}

\begin{remark}
The proof of Lemma \r{w61} will be much easier if $u$ is a classical solution to \eq{wq9}. Indeed, if there exists $x_0\in \rn$ such that $u(x_0)=0$, then
\begin{equation*}
  C(N,s)P.V.\int_{\rn}\frac{0-u(y)}{|x_0-y|^{N+2s}}=(-\De)^su(x_0)+c(x_0) u(x_0)\ge0,
\end{equation*}
which and  $u\ge0$ imply  $u\equiv0$.
\end{remark}

\vspace{0.2cm}
\section{Concentration phenomena of penalized solutions}\label{sec3}
In this section, we aim to prove the concentration of $u_{\varepsilon}$ given in Lemma \r{lem3.6}. We prove that $u_{\varepsilon}$  has a maximum point concentrating at a local minimum of $V$  in $\Lambda$ as $\varepsilon\to 0$. This concentration phenomenon is crucial in linearizing the penalized equation $(\ref{eqs3.2})$. We prove the concentration through comparing energy, in which more regularity results on $u_{\varepsilon}$ will be needed.

Before studying asymptotic behavior of $u_\va$ as $\va\to 0$, we first give some knowledge about the limiting problem of \eqref{eqs3.2}:
\begin{equation}\label{eqs4.1}
(-\Delta)^s u +\lambda u=(\I|u|^p)|u|^{p-2}u, \quad x\in
\R^N,
\end{equation}
where $\lambda>0$ is a constant and $u\in \hs$. The limiting functional $\mathcal{I}_\lambda:\hs\to\R$ corresponding to equation $(\ref{eqs4.1})$ is
$$\mathcal{I}_{\lambda}(u)=\frac{1}{2}\iint_{\R^{2N}}\frac{|u(x)-u(y)|^2}{|x-y|^{N+2s}}+\frac{\lambda}{2}\int_{\R^N}|u|^2-\frac{1}{2p}\int_{\R^N}\big|\i|u|^p\big|^2.$$
By Proposition \r{prop2.4}, $\mathcal{I}_\la$ is well-defined in $\hs$ if $p\in (\frac{N+\a}{N},\frac{N+\a}{N-2s})$. We denote the limiting energy by
\begin{equation}\label{c1}
\mathcal{C}(\lambda):=\inf_{u\in \hs\setminus\{0\}}\sup_{t\geq 0}\mathcal{I}_\lambda(tu).
\end{equation}
 Since $\mathcal{I}_\la(\u)\le \mathcal{I}_\la(u)$ for  $u\in \hs$,  $\mathcal{I}_\la$ is continuous and $C_c^{\wq}(\rn)$ is dense in $\hs$, we deduce that
\begin{equation}\label{c2}
\mathcal{C}(\lambda)=\mathop{\inf}_{u\in \S\setminus\{0\}\atop{u\ge0}}\sup_{t\geq 0}\mathcal{I}_\lambda(tu).
\end{equation}

The following lemma implies the homogeneity of $\mathcal{I}_\lambda$.
\begin{lemma}\label{lem4.4}
 Let $\la>0$, $p\in (\frac{N+\a}{N}, \frac{N+\a}{N-2s})$ and $u\in \hs$, then
$$\mathcal{C}(\lambda)=\lambda^{\frac{\a+2s}{2s(p-1)}-\frac{N-2s}{2s}}\mathcal{C}(1).$$
\end{lemma}
In particular, since $p< \frac{N+\a}{N-2s}$, $\sc(\la)$ is strictly increasing with respect to $\la$.
\begin{proof}
For any $u\in \hs$, we define
$u_\la(x)=\la^{\frac{\a+2s}{4s(p-1)}}u(\la^{\frac{1}{2s}}x).$
A trivial verification shows that $u$ is a critical point of $\mathcal{I}_1$ if and only if $u_\la$ is a critical point of $\mathcal{I}_\la$, then the assertion follows by the definition of $\sc(\la)$.
\end{proof}

In this section, we always assume that ($\mathcal{P}_1$) and ($\mathcal{P}_2$) hold.
By the analysis above, we now give the upper bound of the Mountain-Pass energy $c_{\varepsilon}$.
\begin{lemma}\label{lem4.5}
 It holds
 \begin{equation*}\label{eqs4.4}
\limsup\limits_{\va\to0}\frac{c_{\varepsilon}}{\va^N}\leq\mathcal{C}(V_0).
\end{equation*}
Moreover, there exists a constant $C>0$ independent of $\va\in(0,\va_0)$ such that
\begin{equation}\label{eqs4.5}
\|u_{\va}\|_\va^2\leq C\va^N,
\end{equation}
where $u_{\va}$ is given by Lemma $\ref{lem3.6}$.
\end{lemma}
\begin{proof}
For a nonnegative function $\psi\in C_c^{\infty}(\R^N)\setminus\{0\}$ and $a\in\Lambda$ with $V(a)=V_0$, we define
$$\psi_\epsilon(x):=\psi\Big(\frac{x-a}{\va}\Big).$$
Clearly, $\mathrm{supp}(\psi_{\varepsilon})\subset\Lambda$ for $\varepsilon$ small,  then $G_\va(x,\psi_\va)=\frac{1}{p}|\psi_\va|^p$. Since
\begin{equation*}
\begin{aligned}
&\lim _{\varepsilon \rightarrow 0}  \int_{\mathbb{R}^{N}}V\left(\va x+a)|\psi\right|^{2}
=\int_{\mathbb{R}^{N}}V(a)|\psi|^{2},
\end{aligned}
\end{equation*}
we can  select $T_0>0$ so large  that $\gamma_{\varepsilon}(t):=tT_0\psi_{\varepsilon}\in\Gamma_{\varepsilon}$ and
\begin{align*}
c_{\varepsilon}\leq \max\limits_{t\in[0,1]}J_{\varepsilon}\big(\gamma_{\varepsilon}(t)\big)
=&\va^N\max\limits_{t\in[0,1]}\Big(\frac{1}{2}\iint_{\R^{2N}}\frac{|tT_0\psi(x)-tT_0\psi(y)|^2}{|x-y|^{N+2s}}\\
&+\frac{1}{2}\int_{\R^N}V(\va x+a)|tT_0\psi|^2-\frac{1}{2p}\int_{\R^N}|\i |tT_0\psi|^p|^2\Big)\\
=&\va^N\big(\max\limits_{t\in[0,T_0]}\mathcal{I}_{V(a)}(t\psi)+o_{\varepsilon}(1)\big)\leq\va^N(\sup\limits_{t>0}\mathcal{I}_{V(a)}\big(t\psi)+o_{\varepsilon}(1)\big).
\end{align*}
By  $(\ref{c2})$ and the arbitrariness of $\psi$,   we deduce that
$$\limsup\limits_{\va\to0}\frac{c_{\varepsilon}}{\va^N}\leq\mathop{\inf}_{\psi\in \S\setminus\{0\}\atop{\psi\ge0}}\sup_{t> 0}\mathcal{I}_{V(a)}(t\psi)=\mathcal{C}\big(V(a)\big)=\mathcal{C}(V_0).$$
Besides, it follows from $\eq{f}$ that $\|u_\va\|_\va^2\le C\va^N$ for a constant $C>0$ independent of $\va$.
\end{proof}

The concentration phenomenon of $u_\va$  will be proved by comparing the Mountain-Pass energy $c_{\varepsilon}$ with the limiting energy $\mathcal{C}(V_0)$. One key step  is to verify that the rescaled  function of $u_\va$ does not vanish as $\varepsilon\to 0$, which needs some further regularity estimates on $u_{\varepsilon}$. To this end, we first use Moser iteration to get the uniform global $L^{\infty}$-estimate.

\begin{lemma}\label{w}
Let $\a\in ((N-4s)_+,N)$, $p\in (\frac{N+\a}{N},\frac{N+\a}{N-2s})$ and $u_\va$ be given by Lemma \r{lem3.6}, then it holds
\begin{equation*}\label{eqs4.6}
\|u_\va\|_{L^\wq(\rn)}\le C,
\end{equation*}
where $C>0$ is a constant independent of $\va$.
\end{lemma}
\begin{proof}
Since $u_\va\ge0$ satisfies $\eq{eqs3.2}$ and $G_\va(x,u_\va)\le g_\va(x,u_\va)u_\va$, it follows from $\eq{eqs4.5}$ that
\begin{equation}\label{d8}
\begin{aligned}
\frac{p}{\va^\a}\int_{\mathbb{R}^{N}}|I_{\frac{\alpha}{2}} * G_{\varepsilon}(x, u_\va)|^{2}
\le& \frac{p}{\va^\a}\int_{\mathbb{R}^{N}}\left(I_{\alpha} * G_{\varepsilon}\left(x, u_\varepsilon\right)\right) g_{\varepsilon}\left(x, u_{\varepsilon}\right) u_{\varepsilon}\\
=&\|u_\epsilon \|_\epsilon ^2
\le C\va^N.
\end{aligned}
\end{equation}

Fix any sequence $\{y_\va\}_{\va>0}\subset\rn$ and define $v_\va(y)=u_\va(y_\va+\va y)$ for $y\in\rn$.
It is easy to check that $v_\va \in H_{V_\va}^s(\rn):=\{u\in\Ds \ | \ \int_{\rn}V_\va \u^2<\wq\}$ is a weak solution to the rescaled equation
\begin{equation}\label{d9}
\begin{aligned}
\Fs v_\va+V_\va v_\va=p\big(\I \G_\va(x,v_\va)\big)\g_\va(x,v_\va),
\end{aligned}
\end{equation}
where $V_\va(x)=V(y_\va+\va x)$ and
\begin{align}\label{G}
\G_\va(x,s)=G_\va(y_\va+\va x,s),\ \ \g_\va(x,s)=g_\va(y_\va+\va x,s).\nonumber
\end{align}
Since $V_\va,v_\va\ge0$ and $\g_\va(x,s)\le s_+^{p-1}$, we deduce that  $v_\va$ weakly satisfies
\begin{equation}\label{dd}
\begin{aligned}
\Fs v_\va\le C\big(\I \G_\va(x,v_\va)\big)v_\va^{p-1}.
\end{aligned}
\end{equation}
From $\eq{eqs4.5}$, $\eq{d8}$ and Proposition \r{prop2.2}, by a change of variable, we have
\begin{equation}\label{de}
\begin{aligned}
\int_{\mathbb{R}^{N}}|I_{\frac{\alpha}{2}} * \G_{\varepsilon}\left(x, v_{\varepsilon}\right)|^{2}=\frac{1}{\va^{N+\a}}\int_{\mathbb{R}^{N}}|I_{\frac{\alpha}{2}} * G_{\varepsilon}\left(x, u_{\varepsilon}\right)|^{2}\le C
\end{aligned}
\end{equation}
and
\begin{equation}\label{dff}
\begin{aligned}
&\|v_\va\|^2_{L^{2_s^*}(\rn)}\le C\Big([v_\va]^2_s+\int_{\rn}V_\va v_\va^2\Big)=\frac{C}{\va^N}    \|u_\varepsilon \|_\varepsilon ^2          \le C.
\end{aligned}
\end{equation}
Let $\be\ge1$ and $T>0$. Define
\begin{equation}\label{qdj}
\varphi_{\be,T}(t)=\left\{\begin{array}{l}
0, \text { if } t \leqslant 0, \\
t^{\beta}, \text { if } 0<t<T, \\
\beta T^{\beta-1}(t-T)+T^{\beta}, \text { if } t \geqslant T.
\end{array}\right.
\end{equation}
Since $\varphi_{\be,T}$ is convex and Lipschitz, we see  that
\begin{align}\label{wql}
  \varphi_{\be,T}(v_\va), \varphi_{\be,T}'(v_\va)\ge0\ \mathrm{and}\ \varphi_{\be,T}(v_\va), \varphi_{\be,T}(v_\va)\varphi_{\be,T}'(v_\va)\in H_{V_\va}^s(\rn).
\end{align}
Moreover, $\varphi_{\be,T}(v_\va)$ satisfies   the following inequality
\begin{equation}\label{dg}
\begin{aligned}
\Fs\varphi_{\be,T}(v_\va)\le\varphi_{\be,T}'(v_\va)\Fs v_\va
\end{aligned}
\end{equation}
in the weak sense.
It follows from Proposition \r{prop2.2} that
\begin{align}\label{dk}
\|\varphi_{\be,T}(v_\va)\|^2_{L^{2_s^*}(\rn)}&\le C\int_{\rn}|\fs\varphi_{\be,T}(v_\va)|^2\nonumber\\
&=C\int_{\rn}\varphi_{\be,T}(v_\va)\Fs\varphi_{\be,T}(v_\va)\nonumber\\
&\le C\int_{\rn}\varphi_{\be,T}(v_\va)\varphi_{\be,T}'(v_\va)\Fs v_\va.
\end{align}
Noting the fact that $v_\va\varphi_{\be,T}'(v_\va)\le\be\varphi_{\be,T}(v_\va)$, by $\eq{dd}$, \eq{wql} and \eq{dk}, we obtain that
\begin{align}\label{dl}
&\|\varphi_{\be,T}(v_\va)\|^2_{L^{2_s^*}(\rn)}
\le C\be\int_{\rn}\big(\varphi_{\be,T}(v_\va)\big)^2\big(\I\G_\va(x,v_\va)\big)v_\va^{p-2}:=L_1.
\end{align}
By H\"{o}lder inequality, $\eq{de}$ and Proposition \r{prop2.4}, we have the following estimate on $L_1$:
\begin{align}\label{dp}
L_1\le&C\be\Big(\int_{\rn}\Big|\i \Big(\big(\varphi_{\be,T}(v_\va)\big)^2v_\va^{p-2}\Big)\Big|^2\Big)^{\frac{1}{2}}\Big(\int_{\rn}\left|\i \G_\va(x,v_\va)\right|^2\Big)^{\frac{1}{2}}\nonumber\\
\le& C\be\Big(\int_{\rn}\big(\varphi_{\be,T}(v_\va)\big)^{\frac{4N}{N+\a}}v_\va^{(p-2)\frac{2N}{N+\a}}\Big)^{\frac{N+\a}{2N}}.
\end{align}
Substituting \eq{dp} into \eq{dl}, we conclude that
\begin{align}
\|\varphi_{\be,T}(v_\va)\|^2_{L^{2_s^*}(\rn)}
\le C\be\Big(\int_{\rn}\big(\varphi_{\be,T}(v_\va)\big)^{\frac{4N}{N+\a}}v_\va^{(p-2)\frac{2N}{N+\a}}\Big)^{\frac{N+\a}{2N}}.\nonumber
\end{align}
Letting $T\to\wq$, by Monotone Convergence Theorem, we get
\begin{align}\label{ttog}
  \Big(\int_{\rn}v_\va^{\be2_s^*}\Big)^{\frac{2}{2_s^*}}\le C\be\Big(\int_{\rn}v_\va^{\be\frac{4N}{N+\a}+(p-2)\frac{2N}{N+\a}}\Big)^{\frac{N+\a}{2N}}.
\end{align}
Choosing $\{\be_{i}\}_{i\ge1}$ so that
$$\be_{i+1}\frac{4N}{N+\a}+(p-2)\frac{2N}{N+\a}=\be_i2_s^*,\ \ \beta_0=1,$$
we have
$$
\be_{i+1}+d=\frac{N+\a}{2(N-2s)}(\be_i+d),\quad d=\frac{\frac{p}{2}-1}{1-\frac{1}{2}\frac{N+\a}{N-2s}}> -1,
$$ and
$\frac{N+\a}{2(N-2s)}>1$ by $\a>(N-4s)_+$.

Letting $\be=\be_{i+1}$ in $\eq{ttog}$, we obtain
\begin{equation*}\label{eg}
\begin{aligned}
\Big(\int_{\rn}v_\va^{2_s^*\be_{i+1}}\Big)^\frac{1}{2_s^*(\be_{i+1}+d)}
\le(C\be_{i+1})^\frac{1}{2(\be_{i+1}+d)}\Big(\int_{\rn}v_\va^{2_s^*\be_{i}}\Big)^\frac{1}{2_s^*(\be_{i}+d)}.
\end{aligned}
\end{equation*}
Therefore, by iteration, one gets that
\begin{align}
  \Big(\int_{\rn}v_\va^{2_s^*\be_{i}}\Big)^\frac{1}{2_s^*(\be_{i}+d)}\le\prod_{i=1}^{\infty}(C\be_i)^{\frac{1}{2(\be_i+d)}}
\Big(\int_{\rn}v_\va^{2_s^*}\Big)^\frac{1}{2_s^*(1+d)}\le C,\nonumber
\end{align}
which implies
$(\int_{\rn}v_\va^{2_s^*\be_{i}})^\frac{1}{2_s^*\be_{i}}\le C$ too,
where $C>0$ is some constant independent of $i$ and $\va$. Letting $i\to\wq$, we conclude that
$\|v_\va\|_{L^\infty(\rn)}\le C$
uniformly for $\va$.

By  the definition of $v_\va$,  we complete the proof.
\end{proof}

\begin{remark}
As shown in \cite[Proposition 5]{re} and \cite[Theorem 12.2.1]{clp},  because of the nonlocal nature of
$(-\De)^s$ ($0<s<1$), the H\"{o}lder estimate and Schauder estimate for solutions of fractional equations demand the global $L^\wq$ information instead of local $L^\wq$ information, which is quite different from the classical case $(s=1)$. To ensure a uniform upper bound of $\|u_\va\|_{L^\wq(\rn)}$ for $\va\in(0,\va_0)$, Lemma \r{lem4.5} plays a key role, see \eq{de}-\eq{dff}.
\end{remark}

Now we are going to give the $L^{\infty}$-estimate for the Choquard term.
\begin{lemma}\label{I}
Let $\a\in\big((N-4s)_+,N\big)$, $p\in(\frac{N+\a}{N},\frac{N+\a}{N-2s})$ and $u_\va$ be given by Lemma \r{lem3.6}, then for any sequence $\{y_\va\}_{\va>0}\in\rn$, it holds
$$\|\I\big(\G_\va(x,v_\va)\big)\|_{L^\wq(\rn)}\le C,$$
where $v_\va(y)=u_\va(y_\va+\va y)$, $\G_{\varepsilon}(y,s)=G_\va(y_\va+\va y,s)$, $C>0$ is a constant independent of $\va$ and  $\{y_\va\}_{\va>0}$. 
\end{lemma}

\begin{proof}
From $\eq{dff}$ and Lemma \r{w}, i.e., $\|v_\va\|_{L^{2_s^*}(\rn)}\le C$ and $\|v_\va\|_{L^{\wq}(\rn)}\le C$, we get $\|v_\va\|_{L^{q}(\rn)}\le C$ uniformly for
$\va>0$ and $q\ge2_s^*$. By \eq{3y7}, we have
\begin{align}\label{i}
\I\big(\G_\va(x,v_\va)\big)\le &\I\big(\P(y_\va+\va y)v_\va\big)+\frac{1}{p}\I\big(\x(y_\va+\va y)v_\va^p\big)\nonumber\\
:=&D_1+D_2.
\end{align}

We first estimate $D_1$. By a change of variable, H\"{o}lder inequality, ($\mathcal{P}_2$) and $\eq{eqs4.5}$, we have
\begin{align}\label{D1}
D_1=&\int_{|x-y|\le1}\frac{1}{|x-y|^{N-\a}}\P(y_\va+\va y)v_\va~\d y+\int_{|x-y|>1}\frac{1}{|x-y|^{N-\a}}\P(y_\va+\va y)v_\va~\d y\nonumber\\
\le&\|\P\|_{L^\wq(\rn)}\|v_\va\|_{L^\wq(\rn)}\int_{|x-y|\le1}\frac{1}{|x-y|^{N-\a}}~\d y\nonumber\\
&+\Big(\int_{|x-y|>1}\frac{1}{|x-y|^{2N-2\a}|y_\va+\va y|^\a}~\d y\Big)^{\frac{1}{2}}
\Big(\int_{|x-y|>1}\P^2(y_\va+\va y)v_\va^2|y_\va+\va y|^\a~\d y\Big)^{\frac{1}{2}}\nonumber\\
\le&C_1+\Big(\frac{1}{\va^\a}\int_{|y|>1}\frac{1}{|(\frac{y_\va}{\va}+x)-y|^\a|y|^{2N-2\a}}~\d y\Big)^{\frac{1}{2}}\Big(\frac{1}{\va^N}\int_{\rn}\P^2 u_\va^2|y|^\a\Big)^{\frac{1}{2}}\le C,
\end{align}
where we have used the fact that $\sup_{z\in\rn}\int_{|y|>1}\frac{1}{|z-y|^\a|y|^{2N-2\a}}~\d y\le C.$

Next we estimate $D_2$. By a change of variable, Proposition \r{prop2.3} and $\eq{eqs4.5}$, it holds
\begin{align}\label{D2}
D_2=&\int_{|x-y|\le1}\frac{1}{|x-y|^{N-\a}}\x(y_\va+\va y)v_\va^p~\d y+\int_{|x-y|>1}\frac{1}{|x-y|^{N-\a}}\x(y_\va+\va y)v_\va^p~\d y\nonumber\\
&\le\|v_\va\|_{L^\wq(\rn)}^p\int_{|x-y|\le1}\frac{1}{|x-y|^{N-\a}}~\d y+\frac{1}{\va^N}\int_{\La}|u_\va|^p\le C.
\end{align}
Substituting \eq{D1} and \eq{D2} into \eq{i}, we see that $\|\I\big(\G_\va(x,v_\va)\big)\|_{L^\wq(\rn)}\le C$ uniformly for $\va$.
\end{proof}

\begin{remark}\label{qt0}
The upper energy estimates (Lemma \r{lem4.5}) and the properties of penalization play a very important role in Lemma \r{I} (see \eq{D1}-\eq{D2}).
On the other hand, the regularity  helps us to check Lemma \r{lem4.7} (see \eq{jx1}), which is a significant step to make it possible to realize the desired penalization.
This indicates that the regularity and  the construction of penalization are not mutually independent but interrelated.

\end{remark}

In terms of Lemma \r{w} and Lemma \r{I}, we continue to prove the locally H\"{o}lder estimate of $u_{\varepsilon}$, where the fact $\|u_\va\|_{L^\wq(\rn)}\le C$ in Lemma \r{w} is essential.
\begin{lemma}\label{s}
Let $\a\in\big((N-4s)_+,N\big)$, $p\in(\frac{N+\a}{N},\frac{N+\a}{N-2s})$ and $u_\va$ be given by Lemma \r{lem3.6}, then for any $R>0$ and $\va\in(0,\va_0)$, we have $v_\va\in C^\sigma(B_R(0))$ for any $\sigma\in (0,\min\{2s,1\})$ and
$$\|v_\va\|_{C^\sigma(B_R(0))}\le C(\sigma, N,s,\a, R, y_0),$$
where $C>0$ is independent of $\va$, $v_\va=u_\va(y_\va+\va y)$ such that $y_\va\to y_0$ for some $y_0\in \rn$ as $\varepsilon \to 0$.

If we  assume additionally  that
$V\in L^\wq(\rn)$, then the estimate above is global, i.e., $v_\va\in C^\sigma(\rn)$ and
\begin{equation}\label{g1}
\begin{aligned}
\|v_\va\|_{C^\sigma(\rn)}\le C(\sigma,N,s,\a).
\end{aligned}
\end{equation}

\end{lemma}

 \begin{proof}
 Fix $R>0$ and any $y_*\in B_R(0)$, we have $B_3(y_*)\subset B_{R+3}(0)$. Since $y_\va\to y_0$ as $\varepsilon \to 0$, there exists $R_0>0$ such that $y_\va\in B_{R_0}(y_0)$ for $\va\in(0,\va_0)$.  Denote
 $C_{R,y_0}=\sup_{y\in B_{\tilde{R}}(0)}V(y)$, where $\tilde{R}=R+3+R_0+|y_0|$, we have  $y_\va+B_3(y_*)\subset B_{\tilde{R}}(0)$.

 Recalling $\eq{d9}$ and Lemma \r{w}, we see that $v_\va\in \dot{H}^s(\rn)\cap L^\wq(\rn)$ solves weakly the following equation
\begin{equation}\label{k1}
\begin{aligned}
\Fs v_\va=f_\va,\ \  x\in B_3(y_*),
\end{aligned}
\end{equation}
 where $f_\va:=p\big(\I \G_\va(x,v_\va)\big)\g_\va(x,v_\va)-V_\va v_\va$. By Lemmas \r{w}, \r{I} and the above analysis, it holds that
$f_\va\in L_{\mathrm{loc}}^\wq(\rn)$ and $\|f_\va\|_{L^\wq(B_1(y_*))}\le C+CC_{R,y_0}$.
From Proposition 5 in \c{re}, it follows that $v_\va\in C^\sigma\big(B_{1/4}(y_*)\big)$ for any $\si\in (0,\min\{2s,1\})$ and
\begin{equation}\label{k2}
\begin{aligned}
\|v_\va\|_{C^\sigma(B_{1/4}(y_*))}\le C(\|v_\va\|_{L^\wq(\rn)}+\|f_\va\|_{L^\wq(B_1(y_*))})\le C+CC_{R,y_0},
\end{aligned}
\end{equation}
where $C$ and $C_{R,y_0}$ are independent of $y_*\in B_R(0)$. For any $y_1,y_2\in B_R(0)$ and $y_1\neq y_2$, we have
$y_1,y_2\in B_{1/4}(y_1)$  if $|y_1-y_2|<\frac{1}{4}$. It follows from $\eq{k2}$ that
\begin{equation}\label{k3}
\begin{aligned}
\frac{|v_\va(y_1)-v_\va(y_2)|}{|y_1-y_2|^\si}\le C+CC_{R,y_0}.
\end{aligned}
\end{equation}
If $|y_1-y_2|\ge\frac{1}{4}$, we deduce that
\begin{equation}\label{k4}
\begin{aligned}
\frac{|v_\va(y_1)-v_\va(y_2)|}{|y_1-y_2|^\si}\le 8\|v_\va\|_{L^\wq(\rn)}\le C.
\end{aligned}
\end{equation}
Therefore, by $\eq{k3}$ and $\eq{k4}$, we have
\begin{equation*}\label{k5}
\begin{aligned}
&[v_\va]_{C^\si(B_R(0))}=\mathop{\sup}_{y_1,y_2\in B_R(0)\atop{y_1\neq y_2}}\frac{|v_\va(y_1)-v_\va(y_2)|}{|x-y|^\si}
\le C+CC_{R,y_0}.
\end{aligned}
\end{equation*}
Furthermore, if $V\in L^\wq(\rn)$, then $C_{R,y_0}\le \|V\|_{L^\wq(\rn)}$ and thereby  $\|v_\va\|_{C^\si(\rn)}\le C$. Thus the assertion holds.
 \end{proof}

By the regularity above, now we can give a lower bound on the energy of $u_{\varepsilon}$ by  blow-up analysis.

\begin{lemma}\label{lem4.7}
Let $\a\in ((N-4s)_+,N)$, $p\in (\frac{N+\a}{N},\frac{N+\a}{N-2s})$, $\{\varepsilon_n\}\subset\R_+$ with $\lim\limits_{n\to\infty}\varepsilon_n=0$, $u_n:=u_{\va_n}$ be given by Lemma \r{lem3.6} and $\{(x_n^j)_{n\ge 1}\subset\R^N:1\le j\le k\}$ be $k$ families of points satisfying $\lim\limits_{n\to\infty}x_n^j=x_{\ast}^j$. If the following statements hold
\begin{equation}\label{eqs1b}
V(x_\ast^j)>0,\quad\lim_{n\to\infty}\frac{|x_n^i-x_n^j|}{\varepsilon_n}=\infty\ \ \text{for\ every}\ 1\le i\neq j\le k
\end{equation}
and
\begin{equation}\label{eqs4.17}
\liminf_{n\to\infty}\|u_{n}\|_{L^{\infty}(B_{\varepsilon_n\rho}(x_n^j))}+\va_n^{-\a}\|\I G_{\va_n}(x,u_{n})\|_{L^\wq(B_{\va_n\rho}(x_n^j))}>0
\end{equation}
for   $1\le j\le k$ and some $\rho>0$,
then $x_*^j\in\bar\La$ and
\begin{equation*}\label{eqs4.18}
\liminf_{n\to\infty}\frac{J_{\varepsilon_n}(u_{\varepsilon_n})}{\varepsilon_n^N}\geq\sum_{j=1}^k\mathcal{C}\big(V(x_{\ast}^j)\big),
\end{equation*}
where $\mathcal{C}\big(V(x_{\ast}^j)\big)$ is given by \eqref{c1}.
\end{lemma}

\begin{proof}
The rescaled function $v_n^j$ defined as
$v_n^j(x)=u_{n}(x_n^j+\varepsilon_nx)$
satisfies
\begin{equation}\label{eqs4.20}
(-\Delta)^sv_n^j+V_n^jv_n^j=p\big(\I \mathcal{G}_n^j(v_n^j)\big)\mathfrak{g}_n^j(v_n^j),
\end{equation}
where $V_n^j(x)=V(x_n^j+\varepsilon_nx),\ \mathcal{G}_n^j(v_n^j)=G_{\va_n}(x_n^j+\varepsilon_nx,v_n^j),\ \mathfrak{g}_n^j(v_n^j)=g_{\va_n}(x_n^j+\varepsilon_nx,v_n^j)$.
We also denote the rescaled set $\La_n^j=\{y\in\rn:x_n^j+\va_ny\in\La\}$. Since $\La$ is smooth, up to a subsequence, we can assume that $\chi_{\La_n^j}\to\chi_{\La_*^j}$
a.e. as $n\to\wq$, where $\La_*^j\in\{\rn, H, \emptyset\}$ and $H$ is a half-space in $\rn$.

By Lemma \r{lem4.5}, we have $\|u_n\|_{\va_n}^2\le C\va_n^N$.
A change of variable and Proposition \r{prop2.3} implies that
\begin{eqnarray}\label{h1}
[v_n^j]_s^2+\int_{\rn}V_n^j(v_n^j)^2=\frac{1}{\va_n^N}\|u_n\|_{\va_n}^2\leq C,
\end{eqnarray}
and
\begin{align}\label{wan}
  \int_{\rn}\chi_{\La_n^j}(v_n^j)^p~\d y=\frac{1}{\va_n^N}\int_{\rn}\chi_{\La}u_n^p~\d y\le C.
\end{align}
Moreover, since $\mathcal{G}_n^j(v_n^j)\le \mathfrak{g}_n^j(v_n^j)v_n^j$, by \eq{eqs4.20} and \eq{h1}, we have
\begin{align}\label{qjl}
  p\int_{\rn}\big(\I \mathcal{G}_n^j(v_n^j)\big)\mathcal{G}_n^j(v_n^j)\le[v_n^j]_s^2+\int_{\rn}V_n^j(v_n^j)^2\le C.
\end{align}
Taking a subsequence if necessary, there exists $v_*^j\in\Ds$ such that $v_n^j\rightharpoonup v_*^j$ weakly in $\Ds$, $v_n^j\to v_*^j$ strongly in $L_{\mathrm{loc}}^q(\rn)$ for $q\in [1,2_s^*)$ and  $v_n^j\to v_*^j$ a.e. as $n\to\wq$.
Besides,
$p\mathcal{G}_n^j(v_n^j)\to \chi_{\La_*^j}(v_*^j)^p$ a.e. as $n\to\wq$.

By the weak lower semicontinuity of the norms and Fatou's lemma, we have
\begin{align}\label{pol}
 \int_{\rn}\big(\I(\chi_{\La_*^j}(v_*^j)^p)\big)\chi_{\La_*^j}(v_*^j)^p\le \liminf_{n\to\wq}\int_{\rn}p^2\big(\I \mathcal{G}_n^j(v_n^j)\big)\mathcal{G}_n^j(v_n^j)\leq C,
\end{align}
and
$$[v_*^j]_s^2+\int_{\rn}V(x_*^j)(v_*^j)^2\le \liminf_{n\to\wq}\left([v_n^j]_s^2+\int_{\rn}V_n^j(v_n^j)^2\right)\leq C,$$
which implies that $v_*^j\in \hs$ since $V(x_*^j)>0$. In addition, $v_*^j\ge0$ a.e. in $\rn$ since $v_n^j\ge0$ a.e. in $\rn$.
Moreover, by Proposition \r{prop2.2}, Lemma \r{w} and Lemma \r{s}, we deduce that $v_n^j\to v_*^j$ in $L_{\mathrm{loc}}^q(\rn)$ for any $q\in[1,+\wq]$ as $n\to\wq$ and $\|v_n^j\|_{L^q(\rn)}\le C$ for any $q\in[2_s^*,+\wq]$.

We claim that
\begin{equation}\label{4oo}
p\big(\I \mathcal{G}_n^j(v_n^j)\big)\to\I \big(\chi_{\La_*^j}(v_*^j)^p\big) \ \  in \ \ L_{\mathrm{loc}}^\wq(\rn)\ \  as \ \ n\to\wq.
\end{equation}

 Indeed,
by Fatou's lemma and Lemma \r{I}, we have
\begin{align}\label{4o}
&\|\I \big(\chi_{\La_*^j}(v_*^j)^p\big)\|_{L^\wq(\rn)}\le \sup_{n\in\N}\|p\big(\I \mathcal{G}_n^j(v_n^j)\big)\|_{L^\wq(\rn)}\le C.
\end{align}
For any given $R>1$ and $x\in B_R(0)$, it holds
\begin{equation}\label{4i}
\begin{aligned}
&\big|p\big(\I \mathcal{G}_n^j(v_n^j)\big)-\I \big(\chi_{\La_*^j}(v_*^j)^p\big)\big|\\
\le&\int_{\rn}\frac{1}{|x-y|^{N-\a}}\big|\chi_{\La_n^j}(v_n^j)^p-\chi_{\La_*^j}(v_*^j)^p\big|~\d y+
p\int_{\rn}\frac{1}{|x-y|^{N-\a}}\mathcal{P}_{\va_n}(x_n^j+\varepsilon_ny)v_n^j~\d y.
\end{aligned}
\end{equation}
By H\"{o}lder inequality, \eq{eqs2.1}, ($\mathcal{P}_2$) and $\eq{h1}$, letting
$M>2$, we have
\begin{align}\label{4s}
&\int_{\rn}\frac{1}{|x-y|^{N-\a}}\mathcal{P}_{\va_n}(x_n^j+\varepsilon_ny)v_n^j~\d y\nonumber\\
\le&\|\mathcal{P}_{\va_n}\|_{L^\wq(\rn)}\|v_n^j\|_{L^\wq(\rn)}\int_{|x-y|\le MR}\frac{1}{|x-y|^{N-\a}}~\d y\nonumber\\
&+\Big(\int_{\{|x-y|>MR\}\cap(\La_n^j)^c}\frac{1}{|x-y|^{2N-2\a}|x_n^j+\va_ny|^\a}~\d y\Big)^{\frac{1}{2}}
\Big(\frac{1}{\va^N}\int_{\rn}\mathcal{P}_{\va_n}^2u_n^2|y|^\a~\d y\Big)^{\frac{1}{2}}\nonumber\\
\le& C\|\mathcal{P}_{\va_n}\|_{L^\wq(\rn)}M^\a R^\a+\frac{C}{(MR)^{\frac{N-\a}{4}}}\Big(\int_{\{|x-y|>MR\}\cap(\La_n^j)^c}\frac{1}{|x-y|^{\frac{3}{2}N-\frac{3}{2}\a}|\frac{x_n^j}{\va_n}+y|^\a}~\d y\Big)^{\frac{1}{2}}\nonumber\\
\le&C\|\mathcal{P}_{\va_n}\|_{L^\wq(\rn)}M^\a R^\a+\frac{C}{(MR)^{\frac{N-\a}{4}}}.
\end{align}
On the other hand, by H\"{o}lder inequality, \eq{wan} and $v_*^j\in H^s(\rn)\cap L^\wq(\rn)$, it follows that
\begin{align}\label{4j}
&\int_{\rn}\frac{1}{|x-y|^{N-\a}}\big|\chi_{\La_n^j}(v_n^j)^p-\chi_{\La_*^j}(v_*^j)^p\big|~\d y\nonumber\\
\le&\Big(\int_{|y|\le MR}\frac{1}{|x-y|^{N-\beta}}~\d y\Big)^{\frac{N-\a}{N-\be}}
\Big(\int_{|y|\le MR}\big|\chi_{\La_n^j}(v_n^j)^p-\chi_{\La_*^j}(v_*^j)^p\big|^{\frac{N-\be}{\a-\be}}\Big)^{\frac{\a-\beta}{N-\be}}\nonumber\\
&+\Big(\int_{|y|> MR}\frac{1}{|x-y|^{\ga}}~\d y\Big)^{\frac{N-\a}{\ga}}
\Big(\int_{|y|> MR}(v_*^j)^{p\frac{\ga}{\ga+\a-N}}\Big)^{\frac{\ga+\a-N}{\ga}}\nonumber\\
&+\frac{1}{[(M-1)R]^{N-\a}}\int_{\rn}\chi_{\La_n^j}(v_n^j)^p~\d y\nonumber\\
\le&\Big(\int_{|y|\le (M+1)R}\frac{1}{|y|^{N-\beta}}\Big)^{\frac{N-\a}{N-\be}}
\Big(\int_{|y|\le MR}\big|\chi_{\La_n^j}(v_n^j)^p-\chi_{\La_*^j}(v_*^j)^p\big|^{\frac{N-\be}{\a-\be}}\Big)^{\frac{\a-\beta}{N-\be}}\nonumber\\
&+\frac{C}{[(M-1)R]^{(\ga-N)\frac{N-\a}{\ga}}}+\frac{C}{[(M-1)R]^{N-\a}},
\end{align}
where $0<\be<\a$ and $\ga>N$ such that $p\frac{\ga}{\ga+\a-N}\ge2$. Since $v_n^j\to v_*^j$ in $L_{\mathrm{loc}}^q(\rn)$ as $n\to\wq$ for $q\in[1,+\wq]$, by Dominated Convergence Theorem, we have
\begin{align}\label{4k}
&\lim_{n\to\wq}\int_{|y|\le MR}\big|\chi_{\La_n^j}(v_n^j)^p-\chi_{\La_*^j}(v_*^j)^p\big|^{\frac{N-\be}{\a-\be}}=0
\end{align}
and
\begin{align}\label{4f}
&p\mathcal{G}_n^j(v_n^j)\to\chi_{\La_*^j}(v_*^j)^{p},\ \mathfrak{g}_n^j(v_n^j)\to \chi_{\La_*^j}(v_*^j)^{p-1}\ \mathrm{in}\ L^q(B_R(0))\ \mathrm{for\ any}\ q\ge1.
\end{align}
From $\eq{4i}$--$\eq{4k}$ and \eqref{eqs2.1}, we conclude that
\begin{align}\label{4h}
&\lim_{n\to\wq}\|p\big(\I \mathcal{G}_n^j(v_n^j)\big)-\I \big(\chi_{\La_*^j}(v_*^j)^p\big)\|_{L^\wq(B_R(0))}=0,
\end{align}
which gives (\ref {4oo}).

 Taking any $\var\in C_c^\wq(\rn)$ as a test function in  $\eq{eqs4.20}$ and letting $n\to\wq$, from $\eq{4f}$, $\eq{4h}$ and $v_n^j\rightharpoonup v_*^j$ in $\Ds$, we deduce that $v_*^j$ satisfies
 \begin{align}\label{jx}
\Fs v_*^j+V(x_*^j)v_*^j=\big(\I\chi_{\La_*^j}(v_*^j)^p\big)\chi_{\La_*^j}(v_*^j)^{p-1}.
\end{align}
Since $v_n^j\to v_*^j$ and $p\I\big(\mathcal{G}_n^j(v_n^j)\big)\to\I\big(\chi_{\La_*^j}(v_*^j)^p\big)$ in $L^\wq_{\mathrm{loc}}(\rn)$, from assumption $\eq{eqs4.17}$, we have
\begin{align}\label{jx1}
&\|v_*^j\|_{L^\wq(B_\rho(0))}+\|\I\big(\chi_{\La_*^j}(v_*^j)^p\big)\|_{L^\wq(B_\rho(0))}\nonumber\\
=&\lim_{n\to\wq}\left(\|v_n^j\|_{L^\wq(B_\rho(0))}+p\|\I\big(\mathcal{G}_n^j(v_n^j)\big)\|_{L^\wq(B_\rho(0))}\right)\nonumber\\
=&\lim_{n\to\wq}\big(\|u_{n}\|_{L^{\infty}(B_{\varepsilon_n\rho}(x_n^j))}+p\va_n^{-\a}\|\I G_{\va_n}(x,u_{n})\|_{L^\wq(B_{\va_n\rho}(x_n^j))}\big)>0.
\end{align}
Consequently, $v_*^j\neq0$ and $\La_*^j\neq\emptyset$. In particular, $x_*^j\in \bar\La$.

Define the functional $T^j:H^s(\R^N)\to\R$ associated with equation $\eq{jx}$ as
\begin{align*}T^j(u)&=\frac{1}{2}\iint_{\R^{2N}}\frac{|u(x)-u(y)|^2}{|x-y|^{N+2s}}+\frac{V(x_\ast^j)}{2}\int_{\R^N}|u|^2-\frac{1}{2p}\int_{\R^N}\big|\i(\chi_{\La_*^j}|u|^p)\big|^2.
\end{align*}
Since $ \chi_{\Lambda_\ast^j}\leq1$ and $v_*^j$ is a nontrivial nonnegative solution to equation $\eq{jx}$, it holds
\begin{equation}\label{jxn}
T^j(v_*^j)=\max_{t>0}T^j(tv_*^j)\geq \sup_{t>0}\mathcal{I}_{V(x_\ast^j)}(tv_*^j)\geq\inf_{u\in H^s(\rn)\setminus\{0\}}\sup_{t>0}\mathcal{I}_{V(x_\ast^j)}(tu)=\mathcal{C}\big(V(x_\ast^j)\big).
\end{equation}

Now we begin estimating the energy of $u_n$. Fixing $R>0$, by the assumption $(\ref{eqs1b})$,  we have
$B_{2\varepsilon_nR}(x_n^j)\cap B_{2\varepsilon_nR}(x_n^l)=\emptyset$ if $j\neq l$ for $n$ large enough. Then by Fatou's lemma, $v_n^j\rightharpoonup v_*^j$ in $\Ds$,
$\eq{4f}$, $\eq{4h}$ and \eq{jxn}, we have
\begin{align}\label{eqs48}
&\liminf_{n\to\infty}\frac{1}{\varepsilon_n^N}\Big(\frac{1}{2}\int_{\cup_{j=1}^kB_{\varepsilon_nR}(x_n^j)}\Big(\int_{\R^N}\varepsilon_n^{2s}\frac{|u_{n}(x)-u_{n}(y)|^2}{|x-y|^{N+2s}}\d y\Big)\d x\nonumber\\
&\quad+\frac{1}{2}\int_{\cup_{j=1}^kB_{\varepsilon_nR}(x_n^j)}V(x)u_{n}^2-\frac{p}{2\va_n^\a}\int_{\cup_{j=1}^kB_{\varepsilon_nR}(x_n^j)}\big(\I G_{\varepsilon_n}(x,u_{n})\big)G_{\varepsilon_n}(x,u_{n})\Big)\nonumber\\
=&\liminf_{n\to\infty}\sum_{j=1}^k\Big(\frac{1}{2}\int_{B_R(0)}\Big(\int_{\R^N}\frac{|v_n^j(x)-v_n^j(y)|^2}{|x-y|^{N+2s}}\d y\Big)\d x\nonumber\\
&\quad+\frac{1}{2}\int_{B_R(0)}V_n^j(v_n^j)^2-\frac{1}{2p}\int_{B_R(0)}\big(\I \mathcal{G}_n^j(v_n^j)\big)\mathcal{G}_n^j(v_n^j)\Big)\nonumber\\
\geq&\sum_{j=1}^k\Big(\frac{1}{2}\int_{B_R(0)}\Big(\int_{\R^N}\frac{|v_*^j(x)-v_*^j(y)|^2}{|x-y|^{N+2s}}\d y\Big)\d x\nonumber\\
&\quad+\frac{1}{2}\int_{B_R(0)}V(x_\ast^j)(v_*^j)^2-\frac{1}{2p}\int_{B_R(0)}\big(\I(\chi_{\La_*^j}(v_*^j)^p)\big)\chi_{\La_*^j}(v_*^j)^p\Big)\nonumber\\
\geq&\sum_{j=1}^k\Big(T^j(v_*^j)-\frac{1}{2}\int_{\R^N\setminus B_R(0)}\Big(\int_{\R^N}\frac{|v_*^j(x)-v_*^j(y)|^2}{|x-y|^{N+2s}}\d y\Big)\d x
-\frac{1}{2}\int_{\R^N\setminus B_R(0)}V(x_\ast^j)|v_*^j|^2\Big)\nonumber\\
\ge&\sum_{j=1}^k\mathcal{C}\big(V(x_\ast^j)\big)+o_R(1).
\end{align}

Next we estimate the integral outside the balls above. Let $\eta\in C^{\infty}(\R^N)$ be such that $0\leq\eta\leq1$, $\eta=0$ on $B_1(0)$ and $\eta=1$ on $\R^N\setminus B_2(0)$. Define
$$\psi_{n,R}(x)=\prod_{j=1}^k\eta(\frac{x-x_n^j}{\varepsilon_nR}).$$
Taking $\psi_{n,R}u_{n}$ as a test function to the penalized equation $(\ref{eqs3.2})$, we get
\begin{align}\label{uu}
&\iint_{\R^{2N}}\varepsilon_n^{2s}\psi_{n,R}(x)\frac{|u_{n}(x)-u_{n}(y)|^2}{|x-y|^{N+2s}}
+\int_{\R^N}V\psi_{n,R}u_{n}^2\nonumber\\
&=-\va_n^{2s}\mathcal{R}_n+\frac{p}{\va_n^\a}\int_{\rn}\big(\I G_{\va_n}(x,u_n)\big)g_{\va_n}(x,u_n)\psi_{n,R}u_n,
\end{align}
where
\begin{align}
&\mathcal{R}_n=\iint_{\R^{2N}}\frac{u_n(y)\big(u_n(x)-u_n(y)\big)\big(\psi_{n,R}(x)-\psi_{n,R}(y)\big)}{|x-y|^{N+2s}}.\nonumber
\end{align}
Noting $G_{\va_n}(x,u_n)\le g_{\va_n}(x,u_n)u_n$, it follows from $\eq{uu}$ that
\begin{align}\label{ut}
&\frac{1}{\varepsilon_n^N}\Big(\frac{1}{2}\int_{\R^N\setminus\cup_{j=1}^kB_{\varepsilon_nR}(x_n^j)}\Big(\int_{\R^N}\varepsilon_n^{2s}\frac{|u_{n}(x)-u_{n}(y)|^2}{|x-y|^{N+2s}}\d y\Big)\d x+\frac{1}{2}\int_{\R^N\setminus\cup_{j=1}^kB_{\varepsilon_nR}(x_n^j)}V(x)|u_{n}|^2\nonumber\\
&\ \ -\frac{p}{2\va_n^\a}\int_{\R^N\setminus\cup_{j=1}^kB_{\varepsilon_nR}(x_n^j)}\big(\I G_{\va_n}(x,u_n)\big)G_{\varepsilon_n}(x,u_n)\Big)\nonumber\\
\ge&-\frac{\varepsilon_n^{2s-N}}{2}\mathcal{R}_n+\frac{p}{2\varepsilon_n^{N+\a}}\int_{\R^N\setminus\cup_{j=1}^kB_{\varepsilon_nR}(x_n^j)}\big(\I G_{\va_n}(x,u_n)\big)g_{\va_n}(x,u_n)u_n(\psi_{n,R}-1).
\end{align}
From \eq{4o}, $\eq{4f}$, $\eq{4h}$ and \eq{pol}, we obtain
\begin{align}\label{eqs51}
&\limsup_{n\to\infty}\Big|\frac{p}{\varepsilon_n^{N+\a}}\int_{\R^N\setminus\cup_{j=1}^kB_{\varepsilon_nR}(x_n^j)}\big(\I G_{\va_n}(x,u_n)\big)g_{\va_n}(x,u_n)u_n(\psi_{n,R}-1)\Big|\nonumber\\
\leq&\limsup_{n\to\infty}\sum_{j=1}^k\Big|\frac{p}{\varepsilon_n^{N+\a}}\int_{B_{2\varepsilon_nR}(x_n^j)\setminus B_{\varepsilon_nR}(x_n^j)}\big(\I G_{\va_n}(x,u_n)\big)g_{\va_n}(x,u_n)u_n\Big|\nonumber\\
=&\limsup_{n\to\infty}\sum_{j=1}^kp\int_{B_{2R}\setminus B_{R}}\big(\I \mathcal{G}_n(v_n^j)\big)\mathfrak{g}_n(v_n^j)v_n^j\nonumber\\
=&\sum_{j=1}^k\int_{B_{2R}\setminus B_{R}}\big(\I(\chi_{\La_*^j}(v_*^j)^p)\big)\chi_{\La_*^j}(v_*^j)^p=o_R(1).
\end{align}

It remains to estimate $\mathcal{R}_n$. Noticing
$$|\psi_{n,R}(x)-\psi_{n,R}(y)|\le\sum_{j=1}^k\Big|\eta\Big(\frac{x-x_n^j}{\va_nR}\Big)-\eta\Big(\frac{y-x_n^j}{\va_nR}\Big)\Big|,$$
by H\"{o}lder inequality and scaling, from $\eq{h1}$ we have
\begin{align}\label{ug}
|\mathcal{R}_n|\le&\sqrt{N}\Big(\iint_{\R^{2N}}\frac{|u_n(x)-u_n(y)|^2}{|x-y|^{N+2s}}\Big)^{\frac{1}{2}}
\Big(\iint_{\R^{2N}}\sum_{j=1}^k\frac{|u_n(y)|^2|\eta\big(\frac{x-x_n^j}{\va_nR}\big)-\eta\big(\frac{y-x_n^j}{\va_nR}\big)|^2}{|x-y|^{N+2s}}\Big)^{\frac{1}{2}}\nonumber\\
\le&C\va_n^{N-2s}\sum_{j=1}^k\Big(\iint_{\R^{2N}}\frac{|v_n^j(y)|^2|\eta\big(\frac{x}{R}\big)-\eta\big(\frac{y}{R}\big)|^2}{|x-y|^{N+2s}}\Big)^{\frac{1}{2}}.
\end{align}
Next we estimate the last integral in $\eq{ug}$, which can be divided into four parts.
In the region $B_{2R}(0)\times B_{4R}(0)$, since $|\eta\big(\frac{x}{R}\big)-\eta\big(\frac{y}{R}\big)|\le \frac{C|x-y|}{R}$ and $v_n^j\to v_*^j$ in $L_{\mathrm{loc}}^2(\rn)$, we get
\begin{align}\label{uk}
&\limsup_{n\to\wq}\int_{B_{2R}(0)}|v_n^j(y)|^2~\d y\int_{B_{4R}(0)}\frac{\left|\eta\big(\frac{x}{R}\big)-\eta\big(\frac{y}{R}\big)\right|^2}{|x-y|^{N+2s}}~\d x\nonumber\\
\le& \limsup_{n\to\wq}\frac{C}{R^2}\int_{B_{2R}(0)}|v_n^j(y)|^2~\d y\int_{B_{6R}(0)}\frac{1}{|x|^{N+2s-2}}~\d x\nonumber\\
=&\frac{C}{R^{2s}}\int_{B_{2R}(0)}|v_*^j(y)|^2~\d y=o_R(1).
\end{align}
Similarly, in the region $B_{4R}(0)\times B_{2R}(0)$,
\begin{align}\label{utt}
&\limsup_{n\to\wq}\int_{B_{4R}(0)}|v_n^j(y)|^2~\d y\int_{B_{2R}(0)}\frac{\left|\eta\big(\frac{x}{R}\big)-\eta\big(\frac{y}{R}\big)\right|^2}{|x-y|^{N+2s}}~\d x
=o_R(1).
\end{align}
In the region $B_{2R}(0)\times \big(\rn\setminus B_{4R}(0)\big)$, since $\left|\eta\big(\frac{x}{R}\big)-\eta\big(\frac{y}{R}\big)\right|\le 2$,
\begin{align}\label{ul}
&\limsup_{n\to\wq}\int_{B_{2R}(0)}|v_n^j(y)|^2~\d y\int_{\rn\setminus B_{4R}(0)}\frac{\left|\eta\big(\frac{x}{R}\big)-\eta\big(\frac{y}{R}\big)\right|^2}{|x-y|^{N+2s}}~\d x\nonumber\\
\le& \limsup_{n\to\wq}C\int_{B_{2R}(0)}|v_n^j(y)|^2~\d y\int_{\rn\setminus B_{2R}(0)}\frac{1}{|x|^{N+2s}}~\d x\nonumber\\
=&\frac{C}{R^{2s}}\int_{B_{2R}(0)}|v_*^j(y)|^2~\d y=o_R(1).
\end{align}
In the region $(\rn\setminus B_{4R}(0))\times B_{2R}(0)$, by H\"{o}lder inequality and $\eq{h1}$, we have
\begin{align}\label{ux}
&\limsup_{n\to\wq}\int_{\rn\setminus B_{4R}(0)}|v_n^j(y)|^2~\d y\int_{ B_{2R}(0)}\frac{\left|\eta\big(\frac{x}{R}\big)-\eta\big(\frac{y}{R}\big)\right|^2}{|x-y|^{N+2s}}~\d x\nonumber\\
\le&\limsup_{n\to\wq} CR^N\int_{\rn\setminus B_{4R}(0)}(v_n^j)^2\frac{1}{|y|^{N+2s}}~\d y\nonumber\\
\le&\limsup_{n\to\wq}\Big(CR^N\Big(\int_{\rn}(v_n^j)^{2_s^*}~\d y\Big)^{\frac{N-2s}{N}}\Big(\int_{\rn\setminus B_{4R^{2}}(0)}\frac{1}{|y|^{(N+2s)\frac{N}{2s}}}~\d y\Big)^{\frac{2s}{N}}\nonumber\\
&\qquad+\frac{C}{R^{2s}}\int_{B_{4R^{2}}(0)\setminus B_{4R}(0)}(v_n^j)^2~\d y\Big)\nonumber\\
\le& \frac{C}{R^N}+\frac{C}{R^{2s}}\int_{B_{4R^{2}}(0)\setminus B_{4R}(0)}(v_*^j)^2~\d y=o_R(1).
\end{align}
Thus we conclude  from $\eq{ug}$--$\eq{ux}$ that
\begin{align}\label{uy}
|\mathcal{R}_n|=\va_n^{N-2s}o_R(1).
\end{align}
Putting $\eq{eqs48}$, $\eq{ut}$, $\eq{eqs51}$ and $\eq{uy}$ together and letting $R\to\wq$, we conclude that
\begin{equation*}\label{tt}
\liminf_{n\to\infty}\frac{J_{\varepsilon_n}(u_{n})}{\varepsilon_n^N}\geq\sum_{j=1}^k\mathcal{C}\big(V(x_{\ast}^j)\big).
\end{equation*}
Hence we complete the proof.
\end{proof}

At the end of this section, by comparing the Mountain-Pass energy $c_{\varepsilon}$ in (\ref{Ade2.11}) and the limiting energy in \eqref{c2}, we apply Lemma \ref{lem4.7} to prove that the penalized solution  $u_{\varepsilon}$ concentrates at a local minimum of $V$ in $\Lambda$ as $\va\to0$.

\begin{lemma}\label{jz}Let $\a\in\big((N-4s)_+,N\big)$, $p\in(\frac{N+\a}{N},\frac{N+\a}{N-2s})$ and $u_\va$ be given by Lemma \r{lem3.6}.
Then there exists a family of points $\{x_{\varepsilon}\}_{\va>0}\subset\Lambda$ and $\rho>0$ such that

$({\rm\romannumeral1})~~\liminf\limits_{\va\to0}\|u_{\va}\|_{L^{\infty}(B_{\va\rho}(x_{\varepsilon}))}>0;$

$({\rm\romannumeral2})~\lim\limits_{\va\to0}V(x_{\varepsilon})=V_0;$

$({\rm\romannumeral3})\liminf\limits_{\va\to0}{\rm dist}(x_{\varepsilon}, \Lambda^c)>0;$

$({\rm\romannumeral4})\limsup\limits_{R\to\wq}\limsup \limits_{\va\to0}\|u_{\varepsilon}\|_{L^{\infty}(U\setminus B_{\va R}(x_{\varepsilon}))}+\frac{1}{\va^\a}\|\I G_{\va}(x,u_\va)\|_{L^\wq(U\setminus B_{\va R}(x_\va))}=0.$
\end{lemma}
\begin{proof}
Testing the equation $\eq{eqs3.2}$ by $u_\va$  and applying \eq{3y7} and Young's inequality, we have
\begin{align}\label{y1}
&\int_{\rn}(\va^{2s}|\fs u_\va|^2+Vu_\va^2)=\frac{p}{\va^\a}\int_{\rn}\big(\I G_\va(x,u_\va)\big)g_\va(x,u_\va)u_\va\nonumber\\
\le&\frac{2p}{\varepsilon^{\alpha}} \int_{\mathbb{R}^{N}}|I_{\frac{\alpha}{2}} *(\chi_{\Lambda} u_\va^{p})|^{2}+\frac{2p}{ \varepsilon^{\alpha}} \int_{\mathbb{R}^{N}}|I_{\frac{\alpha}{2}} *(\mathcal{P}_{\varepsilon} u_\va)|^{2},
\end{align}
By Proposition \r{prop2.5} and the assumption ($\mathcal{P}_2$), it holds
\begin{align}\label{y2}
&\frac{p}{ \varepsilon^{\alpha}}\int_{\mathbb{R}^{N}}|I_{\frac{\alpha}{2}} *(\mathcal{P}_{\varepsilon} u_\va)|^{2}
\le \k\int_{\rn}\va^{2s}|\fs u_\va|^2+Vu_\va^2.
\end{align}
Since $\frac{N+\a}{N}<p<\frac{N+\a}{N-2s}$, we choose $1<p'<p$ such that $2<\frac{2Np'}{N+\a}<2_s^*$. By Proposition \r{prop2.3} and Proposition \r{prop2.4},
\begin{align}\label{y3}
\frac{1}{\varepsilon^{\alpha}} \int_{\mathbb{R}^{N}}|I_{\frac{\alpha}{2}} *(\chi_{\Lambda} u_\va^{p})|^{2}
\le &\frac{C}{\va^\a}\Big(\int_{\La}u_\va^{\frac{2Np}{N+\a}}\Big)^\frac{N+\a}{N}
\le\frac{C}{\va^\a}\|u_\va\|_{L^\wq(\La)}^{2p-2p'}\Big(\int_{\La}u_\va^{\frac{2Np'}{N+\a}}\Big)^\frac{N+\a}{N}\nonumber\\
\le&\frac{C}{\va^{(p'-1)N}}\|u_\va\|_{L^\wq(\La)}^{2p-2p'}\Big(\int_{\rn}\va^{2s}|\fs u_\va|^2+Vu_\va^2\Big)^{p'}.
\end{align}
 Substituting \eq{y2}-\eq{y3} into \eq{y1}, by $u_\va\not\equiv0$ and $\eq{eqs4.5}$, we get
\begin{align}\label{y4}
1-2\k\le& \frac{C}{\va^{(p'-1)N}}\|u_\va\|_{L^\wq(\La)}^{2p-2p'}\|u_\va\|_\va^{2(p'-1)}
\le C\|u_\va\|_{L^\wq(\La)}^{2p-2p'}.
\end{align}
Lemma \r{s} means that $u_{\varepsilon}$ is continuous on $\bar\Lambda$, so we can choose $x_{\varepsilon}\in\bar\Lambda$ as a maximum point of $u_{\varepsilon}$ in $\bar{\Lambda}$. It follows from $\k<1/2$ and $\eq{y4}$ that
$$\liminf_{\va\to0}\|u_{\va}\|_{L^{\infty}(B_{\va\rho}(x_{\varepsilon}))}\geq\liminf_{\va\to0}\|u_{\va}\|_{L^{\infty}(\Lambda)}>0.$$
Taking any subsequence $\{x_{\varepsilon_n}\}\subset\{x_{\varepsilon}\}$ such that $\lim\limits_{n\to\infty}x_{\varepsilon_n}=x_{\ast}$, by Lemmas $\ref{lem4.5}$ and $\ref{lem4.7}$ we obtain
$$\mathcal{C}\big(V_0\big)\geq\liminf_{n\to\infty}\frac{J_{\varepsilon_n}(u_{\varepsilon_n})}{\varepsilon_n^N}\geq\mathcal{C}\big(V(x_{\ast})\big).$$
From the assumption $(\mathcal{V})$ and Lemma $\ref{lem4.4}$, there hold $V(x_{\ast})=V_0$ and $x_{\ast}\in \Lambda$. By the arbitrariness of $ \{x_{\varepsilon_n}\}$,  we have $\lim\limits_{\varepsilon\to 0}V(x_{\varepsilon})=V_0$ and then $\liminf\limits_{\varepsilon\to 0}\dist(x_\va, \La^c)>0$.

Finally we prove $(\rm\romannumeral4)$ by contradiction. If (iv) does not hold, then there exist $\{\varepsilon_n\}\subset\R^+$~with ~$\varepsilon_n\to0$ and $\{z_{\varepsilon_n}\}\subset U$ such that
$$\liminf_{n\to\wq}\|u_{\varepsilon_n}\|_{L^{\infty}(B_{\varepsilon_n\rho}(z_{\varepsilon_n}))}+\frac{1}{\va_n^\a}\|\I G_{\va_n}(x,u_{\va_n})\|_{L^\wq(B_{\va_n\rho}(z_{\va_n}))}>0
$$
and
$$\lim\limits_{n\to\infty}\frac{|x_{\varepsilon_n}-z_{\varepsilon_n}|}{\varepsilon_n}=\infty.$$
Since $\bar U$ is compact, we can assume $z_{\varepsilon_n}\to z_\ast\in\bar{U}$, then $V(z_{\ast})\geq V_0>0$. By Lemmas $\ref{lem4.5}$ and $\ref{lem4.7}$ again, we have
$$\mathcal{C}\big(V(x_{\ast})\big)\geq\liminf_{n\to\infty}\frac{J_{\varepsilon_n}(u_{\varepsilon_n})}{\varepsilon_n^N}
\geq\mathcal{C}\big(V(x_{\ast})\big)+\mathcal{C}\big(V(z_{\ast})\big),$$
which is impossible and hence the proof is  completed.
\end{proof}

 \

\section{Recover the original problem}\label{sec4}
In this section, we show that $u_\va$ given by Lemma \r{lem3.6} is indeed a solution to the original problem \eqref{eqs1.1} by comparison principle.
To do this, the first step is to linearize the penalized problem.

{Beforehand, we state some facts and notations used frequently in this section. Let $\{x_{\varepsilon}\}$ be the points given by Lemma \ref{jz}.
By Lemma $\ref {jz}$ (iii), we have
\begin{align}\label{ss}
c_\La\va |x|\le|x_\va+\va x|\le C_\La\va|x|,\ \ x\in\rn\setminus\La_\va,
\end{align}
where $\La_\va:=\{x\mid x_\va+\va x\in\La\}$, $c_\La,C_\La>0$ are some constants depending on $\La$ but independent of $x$ and $\va$.
 Define the rescaled space
\begin{align}
H_{V_\va}^s(\rn):=\Big\{\psi\in \Ds \ \Big | \ \int_{\rn}V_\va\psi^2<\infty\Big\},\nonumber
\end{align}
where $V_{\varepsilon}(x)=V(\va x+x_{\varepsilon})$.
From ($\mathcal{P}_2$), by rescaling,  we have
\begin{align}\label{sg}
p\int_{\rn}|I_{\frac{\a}{2}}*(\tilde{\mathcal{P}}_{\varepsilon}\var)|^2\le \k\int_{\rn}|\fs \var|^2+V_\va|\var|^2,\quad \forall\ \var\in H_{V_\va}^s(\rn),
\end{align}
where $\tilde{\mathcal{P}}_{\varepsilon}(x)=\mathcal{P}_{\varepsilon}(\va x+x_{\varepsilon})$.

We also define the set of test functions for the weak sub(super)-solutions outside a ball
\begin{align}
H_{c,R}^s(\rn):=\left\{\psi\in \Ds, \psi\ge 0\mid \mathrm{supp} \psi\ \mathrm{is\ compact},\ \psi=0\ \mathrm{in}\ B_R(0)\right\}.\nonumber
\end{align}

\begin{proposition}\label{prop4.1}
Let $\a\in\big((N-4s)_+,N\big)$, $p\in[2,\frac{N+\a}{N-2s})$, ($\mathcal{P}_1$)-($\mathcal{P}_2$) hold, $u_\va$ be given by Lemma \r{lem3.6}, $\{x_{\varepsilon}\}_\va$ be the family of points given by Lemma \ref{jz}. Denote $v_{\varepsilon}(\cdot)=u_{\varepsilon}(\varepsilon\cdot + x_{\varepsilon})$,  then there exist $\nu>0$, $R_*>0$ and $\va_{R}>0$ such that
for any given $R>R_*$ and $\va\in (0,\va_{R})$, $v_{\varepsilon}$ is a weak sub-solution to the following equation
\begin{equation}\label{q1}
(-\Delta)^s v +\frac{1}{2}V_{\varepsilon}v= \big(p\I(\tilde{\P}v)+\nu\va^{N-\a}I_{\a,\va}\big)\tilde{\P},\ x\in\R^N\setminus B_R(0),
\end{equation}
i.e.,
\begin{align}\label{hk}
\int_{\rn}\fs v_\va\fs\var+\frac{1}{2}V_\va v_\va\var\le\int_{\rn}\big(p\I(\tilde{\P}v_\va)+\nu\va^{N-\a}I_{\a,\va}\big)\tilde{\P}\var,
\end{align}
for all $\varphi\in H_{c,R}^s(\rn)$, where~$V_{\varepsilon}(x)=V(\va x+x_{\varepsilon})$,~$\tilde{\mathcal{P}}_{\varepsilon}(x)=\mathcal{P}_{\varepsilon}(\va x+x_{\varepsilon})$, $I_{\a,\va}=I_\a(\va x+x_{\varepsilon})$.
\end{proposition}

\begin{proof}
By Lemma \r{jz}, since $p\ge2$, there exists $R_*>0$ and $\va_{R}>0$ such that
\begin{align}\label{edig}
 &p\va^{-\a}\big(\I G_\va(x,u_\va)\big)u_\va^{p-2}\le \frac{1}{2} V_0\quad \mathrm{in}\ \  U\setminus B_{\va R}(x_\va)
\end{align}
for any $R>R_*$ and $0<\va<\va_{R}$.

Fix  $\varphi\in H_{c,R}^s(\rn)$. Taking $\varphi_{\varepsilon}(\cdot)=\varphi (\frac{\cdot - x_{\varepsilon}}{\varepsilon})$ as a test function in  $\eq{eqs3.2}$ for $u_\va$, namely
\begin{align}\label{ephj}
 \va^{2s}\int_{\rn}\fs u_\va\fs\varphi_{\varepsilon}+Vu_\va\varphi_{\varepsilon}=\int_{\rn}p\va^{-\a}\left(I_{\alpha} * G_{\varepsilon}(x, u_\va)\right) g_{\varepsilon}(x, u_\va)\varphi_{\varepsilon}.
\end{align}
By $g_\va(x,u_\va)\le u_\va^{p-1}$, \eq{edig} and $\inf_{U}V=V_0$, we have
\begin{align}\label{z1}
p\va^{-\a}\big(\I G_\va(x,u_\va)\big)g_\va(x,u_\va)\le p\va^{-\a}\big(\I G_\va(x,u_\va)\big)u_\va^{p-1}\le\frac{1}{2} Vu_\va\ \mathrm{in}\ U\setminus
B_{\va R}(x_\va).
\end{align}
Moreover, by \eq{3y7}, we have
\begin{align}\label{z2}
p\va^{-\a}\big(\I G_\va(x,u_\va)\big)g_\va(x,u_\va)\le p\va^{-\a}\Big(\I \big(\P u_\va+\frac{1}{p}\chi_\La u_\va^p\big)\Big)\P\ \mathrm{in}\ \rn\setminus U.
\end{align}
Since $\dist(\La,\partial U)>0$, $p\ge2$, by Proposition \r{prop2.3} and $\eq{eqs4.5}$, we have
\begin{align}\label{z3}
\va^{-\a}\I\big(\chi_\La u_\va^p\big)\le C\frac{I_\a}{\va^\a}\int_{\La}u_\va^p\le C'I_\a\va^{N(1-\frac{p}{2})-\a}\|u_\va\|_\va^{\frac{p}{2}}\le \nu I_\a\va^{N-\a}\ \mathrm{in}\ \rn\setminus U,
\end{align}
where $\nu>0$ is independent of $R$ and $\va$.

Note that $\varphi_\va=0\ \mathrm{in}\ B_{\va R}(x_\va)$. Substituting \eq{z1}--\eq{z3} into \eq{ephj}, we get
\begin{align}
  \va^{2s}\int_{\rn}\fs u_\va\fs\varphi_{\varepsilon}+\frac{1}{2}Vu_\va\varphi_{\varepsilon}\le\int_{\rn}\left(p \varepsilon^{-\alpha} I_{\alpha} *\left(\mathcal{P}_{\varepsilon} u_{\varepsilon}\right)+\nu \varepsilon^{N-\alpha} I_{\alpha}\right) \mathcal{P}_{\varepsilon}\varphi_{\varepsilon}.\nonumber
\end{align}
Therefore,  it follows by scaling that
\begin{align}
\int_{\rn}\fs v_\va\fs\varphi+\frac{1}{2}V_\va v_\va\varphi\le\int_{\rn}\big(p\I(\tilde{\P}v_\va)+\nu\va^{N-\a}I_{\a,\va}\big)\tilde{\P}\varphi.\nonumber
\end{align}
The conclusion then follows by the arbitrariness of $\varphi$.
\end{proof}

Next, we establish the comparison principle:
\begin{proposition}\label{4z}(Comparison principle)
Let $(\mathcal{P}_2)$ hold and $v\in \dot{H}^s(\rn)$ with
$\int_{\rn}V_\va v_+^2<\wq$.
If $v$ satisfies weakly
\begin{align}\label{n1}
(-\Delta)^s v +\frac{1}{2}V_{\varepsilon}v\leq p\big(\I(\tilde{\P}v)\big)\tilde{\P}\ \mathrm{in}\ \R^N\setminus B_R(0),
\end{align}
and $v\le0$ in $B_R(0)$, then $v\le0$ in $\rn$.
\end{proposition}
\begin{proof}
Clearly, $v_+=0$ in $B_R(0)$ and $v_+\in \dot{H}^s(\rn)$. Then there exists $\{\var_n\}_{n\ge1}\subset H_{c,R}^s(\rn)$ such that
$\var_n\to v_+$ in $\dot{H}^s(\rn)$ as $n\to\wq$. Indeed, by \cite[Lemma 5]{pap}, we can choose $\var_n=\eta(\frac{x}{n})v_+$ where $\eta\in C_c^\wq(\rn,[0,1])$ satisfying
$\eta\equiv1$ in $B_R(0)$ and $\mathrm{supp}\eta\subset B_{2R}(0)$.

Taking $\var_n$ as a test function into $\eq{n1}$, since $\tilde{\P}v\le \tilde{\P}v_+$, we see that
\begin{align}\label{wbj}
&\int_{\rn}\fs v_+\fs\var_n+\frac{1}{2}V_\va v_+\var_n\le p\int_{\rn}\big(\I(\tilde{\P}v_+)\big)\tilde{\P}\var_n,
\end{align}
where we have used that
\begin{align}\label{23j}
  \int_{\rn}\fs v_+\fs\var_n\le \int_{\rn}\fs v\fs\var_n.
\end{align}
Since $\var_n\to v_+$ in $\dot{H}^s(\rn)$ as $n\to\wq$, it follows that
\begin{align}\label{24j}
  \lim_{n\to\wq}\int_{\rn}\fs v_+\fs\var_n=\int_{\rn}|\fs v_+|^2.
\end{align}
Clearly, since $\var_n\le v_+$,
$$\int_{\rn}\big(\I(\tilde{\P}v_+)\big)\tilde{\P}\var_n\le\int_{\rn}\big(\I(\tilde{\P}v_+)\big)\tilde{\P}v_+= \int_{\rn}|\i(\tilde{\P}v_+)|^2.$$
Moreover, by Fatou's Lemma,
\begin{align}\label{25j}
 \int_{\rn}V_\va |v_+|^2\le\liminf_{n\to\wq} \int_{\rn}V_\va v_+\var_n.
\end{align}
Therefore, recalling \eq{wbj} and letting $n\to\wq$, from Proposition
\r{prop2.5} and \eq{sg}, we get
\begin{align}
\int_{\rn}|\fs v_+|^2+\frac{1}{2}V_\va v_+^2\le& p\int_{\rn}|\i(\tilde{\P}v_+)|^2\nonumber\\
\le&\k\int_{\rn}|\fs v_+|^2+V_\va v_+^2,\nonumber
\end{align}
which  implies $v_+=0$ since $v_+\in H_{V_\va}^s(\rn)$ and $\k<1/2$.
\end{proof}

Now we construct the super-solutions for the linear penalized problem $\eq{q1}$. The sup-solutions are selected as
\begin{align}\label{sn}
w_\mu=\frac{1}{(1+|x|^2)^\frac{\mu}{2}},
\end{align}
which belongs to $C^{k,\beta}(\rn)$ for any $k\in \N$ and $\beta\in(0,1)$. Particularly, $\Fs w_\mu$ is well-defined pointwise.

The following two propositions for estimating the nonlocal term $\Fs w_\mu$ are given by our other paper \cite {dpy2023}.

\begin{proposition}\label{tb}
For any $\mu\in(0,+\infty)$, there exists constants $R_\mu, C_\mu, \tilde{C}_\mu>0$ depending only on $\mu$, $N$ and $s$ such that
\begin{align}
\left\{
  \begin{array}{ll}
    0<C_\mu\ds\frac{1}{|x|^{\mu+2s}}\le\Fs w_\mu\le3C_\mu\frac{1}{|x|^{\mu+2s}}, & \mathrm{if }\ |x|>R_\mu\ \mathrm{and}\ \mu\in(0,N-2s); \vspace{1mm}\\
    \Fs w_\mu=C_{N-2s}w_\mu^{2_s^*-1},\ x\in \rn,& \mathrm{if }\ \mu=N-2s;\vspace{1mm}\\
-3C_\mu\ds\frac{1}{|x|^{\mu+2s}}\le\Fs w_\mu\le -C_\mu\frac{1}{|x|^{\mu+2s}}<0,& \mathrm{if }\ |x|>R_\mu\ \mathrm{and}\ \mu\in(N-2s,N);\vspace{1mm}\\
   -\ds\frac{\tilde{C}_{N}\ln|x|}{|x|^{N+2s}}\le\Fs w_\mu\le-\frac{C_{N}\ln|x|}{|x|^{N+2s}}<0,& \mathrm{if }\ |x|>R_\mu\ \mathrm{and}\ \mu=N,\vspace{1mm}\\
   -\ds\frac{\tilde{C}_\mu}{|x|^{N+2s}}\le\Fs w_\mu\le -\frac{C_\mu}{|x|^{N+2s}}<0, & \mathrm{if }\ |x|>R_\mu\ \mathrm{and}\ \mu>N.
\end{array}
\right.\nonumber
\end{align}

\end{proposition}


\begin{proposition}\label{cc5} 
$w_\mu\in \dot{H}^s(\rn)$ for $\mu>\frac{N-2s}{2}$ and $w_\mu\notin \dot{H}^s(\rn)$ for $0<\mu\le\frac{N-2s}{2}$. Moreover,
for any $\mu>\frac{N-2s}{2}$,
\begin{align*}
  \int_{\rn}(-\De)^{s/2}w_\mu(-\De)^{s/2}\phi=\int_{\rn}(-\De)^{s}w_\mu\phi,\quad \forall \phi\in \dot{H}^s(\rn).
\end{align*}
\end{proposition}

\vspace{0.2cm}

Now we are in a position to construct the super-solutions of (\ref {q1}). We  assume the prescribed form of the penalization:
\begin{align}\label{ph}
\P(x)=\frac{\va^\theta}{|x|^\tau}\chi_{\La^c},
\end{align}
where $\theta, \tau>0$ are two parameters which will be determined later. Moreover, in order to described the following proof conveniently, we   give
some notations as follows:
\begin{equation}\label{fd}F^\va_{\theta,\tau, \mu}(x):=
\frac{\va^{2\theta-2\tau}\chi_{\La_\va^c}}{|x|^{\mu+2\tau-\a}}+\frac{(\va^{2\theta-2\tau}+\va^{\theta-\tau})\ln(|x|+e)\chi_{\La_\va^c}}{|x|^{N-\a+\tau}},
\quad \mu+\tau>\a.
  \end{equation}
and
\begin{equation}\label{fdd}G^\va_{\mu}(x):=\left\{
  \begin{aligned}
&\frac{\chi_{\La_\va^c}}{|x|^{\mu+2s}},\ \mathrm{if}\ \mu\in\big(\frac{N-2s}{2},N-2s\big),\\
&\frac{\va^{-2s}\chi_{\La_\va^c}}{|x|^{\mu+2s}},\ \mathrm{if}\ \mu\in(N-2s,N),\ \inf_{\rn}V(x)(1+|x|^{\om})>0\ \mathrm{with}\ \om=2s,\\
&\frac{\va^{-\omega}\chi_{\La_\va^c}}{|x|^{\mu+\omega}},\ \mathrm{if}\ \mu\in(N,N+2s-\omega),\ \inf_{\rn}V(x)(1+|x|^{\omega})>0, \  \omega\in (0,2s).
  \end{aligned}\right.
  \end{equation}

\begin{proposition}\label{5b}(Construction of sup-solutions)
Let
$$
\mu\in\Big(\frac{N-2s}{2},N+2s\Big)\Big\backslash\{N,N-2s\},\,\,\mu+\tau>\a,
$$
and $\{x_{\varepsilon}\}_\va$ be the family of points given by Lemma \ref{jz}.
If $F^\va_{\theta,\tau, \mu}\le \lambda G^\va_{\mu}$ for  given $\lambda>0$ and $\va$ small depending on $\lambda$ , then $w_\mu$ is a supper-solution of (\ref {q1}) in the classical sense, i.e.
\begin{align}\label{wd}
\Fs w_\mu+\frac{1}{2}V_\va w_\mu\ge\big(p\I(\tilde{\P}w_\mu)+\nu\va^{N-\a}I_{\a,\va}\big)\tilde{\P},\ x\in\R^N\setminus B_R(0)
\end{align}
for given $R>0$ large enough, where $\tilde{\P}(x)=\P(x_\va+\va x)$, $I_{\a,\va}(x)=I_\a(x_\va+\va x)$, $V_\va(x)=V(x_\va+\va x)$.
\end{proposition}
\begin{proof}
We first consider the right hand side of $\eq{wd}$. For  given $R>\max\{R_\mu,1\}$, since $\liminf_{\va\to0}\dist(x_\va, \La^c)>0$, we have $B_R(0)\subset \La_\va:=\{x\mid x_\va+\va x\in \La\}$ for small $\va$. Reviewing $\eq{ss}$,
we have
\begin{align}\label{al}
&\big(p\I(\tilde{\P}w_\mu)+\nu\va^{N-\a}I_{\a,\va}\big)\tilde{\P}\nonumber\\
\le&\frac{p}{c_\La^{2\tau}}\va^{2\theta-2\tau}\Big(\I\Big(\frac{\chi_{B^c_1(0)}}{|x|^{\mu+\tau}}\Big)\Big)\frac{\chi_{\La_\va^c}}{|x|^\tau}
+\frac{\nu}{c_\La^{N-\a+\tau}}\va^{\theta-\tau}\frac{\chi_{\La_\va^c}}{|x|^{N-\a+\tau}}.
\end{align}
There exists a constant $C>0$ such that for $\mu+\tau>\a$,
\begin{equation}\label{df}\Big(\I\big(\frac{\chi_{B^c_1(0)}}{|x|^{\mu+\tau}}\big)\Big)(x)\le \frac{C}{|x|^{\mu+\tau-\a}}+\frac{C\ln(|x|+e)}{|x|^{N-\a}},\quad x\in\rn\backslash\{0\}.
  \end{equation}
Indeed, for any $x\in\rn\backslash\{0\}$, we have
\begin{align*}
  &\int_{\rn}\frac{1}{|x-y|^{N-\a}}\frac{\chi_{B^c_1(0)}(y)}{|y|^{\mu+\tau}}dy\\
  =&\int_{B_{|x|/2}(x)}\frac{1}{|x-y|^{N-\a}}\frac{\chi_{B^c_1(0)}(y)}{|y|^{\mu+\tau}}dy+
  \int_{B_{|x|/2}(0)}\frac{1}{|x-y|^{N-\a}}\frac{\chi_{B^c_1(0)}(y)}{|y|^{\mu+\tau}}dy\\
  &+\int_{B^c_{|x|/2}(x)\cap B^c_{|x|/2}(0)}\frac{1}{|x-y|^{N-\a}}\frac{\chi_{B^c_1(0)}(y)}{|y|^{\mu+\tau}}dy\\
  \le&\frac{C}{|x|^{\mu+\tau}}\int_{B_{|x|/2}(x)}\frac{1}{|x-y|^{N-\a}}dy+\frac{C}{|x|^{N-\a}}\int_{1\le |y|\le \frac{|x|}{2}}\frac{1}{|y|^{\mu+\tau}}dy\\
  &+C\int_{B^c_{|x|/2}(0)}\frac{1}{|y|^{N-\a+\mu+\tau}}dy\\
  \le&\frac{C}{|x|^{\mu+\tau-\a}}+\frac{C}{|x|^{N-\a}}\Big(1+\frac{1}{|x|^{\mu+\tau-N}}+\ln(|x|+e)\Big),
\end{align*}
where we use that $|x-y|\ge \frac{1}{3}|y|$ if $y\in B^c_{|x|/2}(x)\cap B^c_{|x|/2}(0)$. Then \eq{df} holds.

Recalling the definition of $F^\va_{\theta,\tau, \mu}$ in \eq{fd}, we infer from $\eq{al}$ and $\eq{df}$ that
\begin{align}
\big(p\I(\tilde{\P}w_\mu)+\nu\va^{N-\a}I_{\a,\va}\big)\tilde{\P}\le CF^\va_{\theta,\tau, \mu}.\nonumber
  \end{align}

Now we consider the left hand side of $\eq{wd}$  in different decay rates of $V$ stated in  \eqref{fdd}.

\noindent\textbf{Case 1.} $\mu\in (\frac{N-2s}{2}, N-2s)$.

 From Proposition \r{tb}, we have $\Fs w_\mu\ge\frac{C}{|x|^{\mu+2s}}$ for $|x|>R$.

\vspace{0.1cm}

\noindent\textbf{Case 2.} $\inf_{x\in\rn}V(x)(1+|x|^{2s})>0$ and $N-2s<\mu<N$.

From Proposition \r{tb}, for $R$ large, we have
$$\Fs w_\mu+\frac{1}{2}V_{\va} w_\mu\ge -\frac{3C_\mu}{|x|^{2s}}\frac{1}{|x|^\mu}+\frac{1}{2}V_0\frac{1}{(1+|x|^2)^{\mu/2}}\ge0,\ \ x\in \La_\va\setminus B_R(0).$$
Since $\inf_{x\in\rn}V(x) (1+|x|^{2s})>0$, there exists $C>0$ such that $V(x)\ge \frac{C}{|x|^{2s}}$ for $|x|\ge1$.
By $\eq{ss}$,  for $\va>0$ small,
 we have
\begin{eqnarray*}
\Fs w_\mu+\frac{1}{2}V_{\va} w_\mu&\ge&-\frac{3C_\mu}{|x|^{2s+\mu}}+\frac{C}{2C_\La^{2s}\va^{2s}}
\frac{1}{|x|^{2s}}\frac{1}{(1+|x|^2)^{\mu/2}}\\
&\ge&\frac{C\va^{-2s}}{|x|^{\mu+2s}},\quad x\in\R^N\backslash\Lambda_\va.
\end{eqnarray*}

\noindent\textbf{Case 3.} $\inf_{x\in\rn}V(x)(1+|x|^{\omega})>0$ for some $\omega\in (0,2s)$ and $N<\mu< N+2s-\omega$.

From Proposition \r{tb}, we get for $R$ large and $\va$ small that
$$\Fs w_\mu+\frac{1}{2}V_{\va} w_\mu\ge -\frac{\tilde{C}_\mu}{|x|^{N+2s}}+\frac{1}{2}V_0\frac{1}{(1+|x|^2)^{\mu/2}}\ge0,\quad x\in\La_\va\setminus B_R(0).$$
Since $\inf_{x\in\rn}V(x) (1+|x|^{\om})>0$, there exists $C_\om>0$ such that $V(x)\ge \frac{C_\om}{|x|^{\om}}$ for $|x|\ge1$. Thus for $\va>0$ small,
it follows by $\eq{ss}$ and Proposition \r{tb} that
$$\Fs w_\mu+\frac{1}{2}V_{\va} w_\mu\ge-
\frac{\tilde{C}_\mu}{|x|^{N+2s}}+\frac{C_\omega}{2C_\La^{\omega}\va^{\omega}}\frac{1}{|x|^{\om}}\frac{1}{(1+|x|^2)^{\mu/2}}\ge\frac{C\va^{-\omega}}{|x|^{\mu+\omega}}$$
for all $x\in\rn\setminus \La_\va$.

Summarizing the three cases above, the conclusion follows by the assumption
$F^\va_{\theta,\tau, \mu}\le \lambda G^\va_{\mu}$ for $\lambda$ small.
\end{proof}

\begin{remark}
Note that there is no restrictions on $V$ out set of $\La$ in case 1,  which indicates that $V$ will not have influence outside $\Lambda$ during the construction in case 1.
However, if $V$ further satisfies $\inf_{x\in\rn}V(x)(1+|x|^{\omega})>0$ for $\omega\in(0,2s]$, we are able to take $\mu> N-2s$ due to the effect of $V$. More precisely, $\Fs w_\mu$ can be absorbed by  $V_{\va} w_\mu$ outside $\La_\va$.
\end{remark}

Next, by means of the sup-solutions above, we are going to apply the comparison principle in Proposition \ref{4z} to prove Theorem $\ref{thm1.1}$.
We need to verify firstly that the two pre-assumptions ($\mathcal{P}_1$)-($\mathcal{P}_2$) in Section \ref{sec2} hold under some choices of the parameters $\tau,\theta$.

\begin{proposition}\label{48}
Assume that one of the following two conditions holds:

 ($\mathcal{S}_1$) $\a+2s<2\tau$, $\a+2s<2\theta$;

 ($\mathcal{S}_2$) $\a<2\theta$ and $\a+\omega<2\tau$ when
$\inf_{x\in\rn}V(x)(1+|x|^\om)>0$ with $\om\in(0,2s]$.

\noindent Then
the penalized function $\mathcal{P}_{\varepsilon}$ defined by  (\ref {ph})  satisfies ($\mathcal{P}_1$) and ($\mathcal{P}_2$) in Section \ref{sec2}.
\end{proposition}

\begin{proof}
We first verify ($\mathcal{P}_2$).

 {\bf The case under the assumption ($\mathcal{S}_1$):} For any $\var\in \Ds$, by the assumption $\a+2s<2\tau$, $\a+2s<2\theta$ and Hardy inequality (Proposition \r{prop2.1}), for $\va$ small we have,
\begin{align}\label{hp}
\frac{pC_\a}{\va^\a}\int_{\rn}\P^2|\var|^2|x|^\a=&pC_\a\va^{2\theta-\a}\int_{\La^c}\frac{1}{|x|^{2\tau-\a}}|\var|^2\nonumber\\
\le& C\va^{2\theta-\a}\int_{\La^c}\frac{|\var|^2}{|x|^{2s}}\le\k\int_{\rn}\va^{2s}|\fs\var|^2,
\end{align}
which implies $(\mathcal{P}_2)$.

{\bf The case under the assumption ($\mathcal{S}_2$):} Clearly, there exists a $C_\omega>0$ such that $V\ge\frac{C_\omega}{|x|^\omega}$ in $\rn\setminus\La$. By the assumptions $\a<2\theta$ and $\a+\omega<2\tau$, for $\va$ small,  we have
\begin{align}
\frac{pC_\a}{\va^\a}\int_{\rn}\P^2|\var|^2|x|^\a\le&pC_\a\va^{2\theta-\a}\int_{\La^c}\frac{1}{|x|^{2\tau-\a}}|\var|^2\nonumber\\
\le&\k C_\om\int_{\La^c}\frac{1}{|x|^\omega}|\var|^2\le\k\int_{\rn} V|\var|^2,\nonumber
\end{align}
which also implies $(\mathcal{P}_2)$.

 Next we turn to check ($\mathcal{P}_1$) . Let $\{v_n\}_{n\in\N}$ is a bounded sequence in $\hv$. Up to a subsequence, there exists
some $v\in\hv$ such that $v_n\rightharpoonup v$ in $\hv$ and $v_n\to v$ in $L_{\mathrm{loc}}^q(\rn)$ for $q\in [1,2_s^*)$. Let $\va<1$ and $M>1$ such that
$\La\subset B_M(0)$.

{\bf The case under the assumption ($\mathcal{S}_1$):} By the assumption $\a+2s<2\tau$, $\a+2s<2\theta$ and Hardy inequality,
\begin{align}\label{zr}
&\int_{\rn}|v_n-v|^2\P^2|x|^\a\nonumber\\
=&\va^{2\theta}\int_{\rn\setminus B_M(0)}\frac{|v_n-v|^2}{|x|^{2\tau-\a}}+\va^{2\theta}\int_{B_M(0)\setminus\La}\frac{|v_n-v|^2}{|x|^{2\tau-\a}}\nonumber\\
\le&\frac{\va^{2s}}{M^{2\tau-\a-2s}}\int_{\rn\setminus B_M(0)}\frac{|v_n-v|^2}{|x|^{2s}}+C\int_{B_M(0)\setminus\La}|v_n-v|^2,\nonumber\\
\le&\frac{C}{M^{2\tau-\a-2s}}\big(\sup_{n\in\N}\va^{2s}[v_n]_s^2+\va^{2s}[v]_s^2\big)+C\int_{B_M(0)\setminus\La}|v_n-v|^2,
\end{align}
which implies $v_n\to v$ in $L^2\big(\rn, \P^2|x|^\a\d x\big)$ as $n\to\wq$ and thereby ($\mathcal{P}_1$) holds.

{\bf The case under the assumption ($\mathcal{S}_2$):} Noting $V\ge\frac{C_\omega}{|x|^\omega}$ in $\rn\setminus\La$, by the assumption $\a+\omega<2\tau$,
\begin{align}\label{zr}
&\int_{\rn}|v_n-v|^2\P^2|x|^\a\nonumber\\
=&\va^{2\theta}\int_{\rn\setminus B_M(0)}|v_n-v|^2\frac{1}{|x|^{2\tau-\a}}+\va^{2\theta}\int_{B_M(0)\setminus\La}|v_n-v|^2\frac{1}{|x|^{2\tau-\a}}\nonumber\\
\le&\frac{1}{M^{2\tau-\a-\omega}}\int_{\rn\setminus B_M(0)}\frac{|v_n-v|^2}{|x|^{\omega}}+C\int_{B_M(0)\setminus\La}|v_n-v|^2,\nonumber\\
\le&\frac{C}{M^{2\tau-\a-\omega}}(\sup_{n\in\N}\int_{\rn}V|v_n|^2+\int_{\rn}V|v|^2)+C\int_{B_M(0)\setminus\La}|v_n-v|^2,
\end{align}
which indicates $v_n\to v$ in $L^2\big(\rn, \P^2|x|^\a\d x\big)$ as $n\to\wq$ and so ($\mathcal{P}_1$) holds.

Then we complete the proof.
\end{proof}

Secondly, we use the comparison principle in Proposition \ref{4z} to get the upper decay estimates of $u_{\varepsilon}$.

\begin{proposition}\label{49}
Let  $\a\in \big((N-4s)_+,N\big)$, $p\in [2,\frac{N+\a}{N-2s})$. Assume that one of the following three conditions holds:

 ($\mathcal{U}_1$) $2s<2\tau-\a$ and $\a<\tau<\theta$, $\mu\in (\frac{N-2s}{2},N-2s)$;

 ($\mathcal{U}_2$) $\a+2s<\tau$ and $\tau<\theta$, $\mu\in (N-2s,N)$, when
$\inf_{x\in\rn}V(x)(1+|x|^{2s})>0$;

 ($\mathcal{U}_3$) $\a+2s<\tau$ and $\tau<\theta$, $\mu\in(N,N+2s-\om)$, when
$\inf_{x\in\rn}V(x)(1+|x|^\om)>0$ with $\om\in(0,2s)$.

\noindent Then ($\mathcal{P}_1$)-($\mathcal{P}_2$) hold and there exists $C>0$ independent of small $\va$  such that
$v_\va:=u_\va(x_\va+\va x)\le Cw_\mu$. In  particular,
\begin{align}\label{wsq}
u_\va\le\frac{C\va^\mu}{|x|^\mu}\ \mathrm{in}\ \rn\setminus\La,
\end{align}
where  $u_\va$ is given by Lemma \r{lem3.6} and $\{x_{\varepsilon}\}_\va$ is given by Lemma \ref{jz}.
\end{proposition}
\begin{proof}

 It is easy to check that ($\mathcal{S}_1$) holds under the assumption ($\mathcal{U}_1$), and     ($\mathcal{S}_2$) holds under one of ($\mathcal{U}_2$) and ($\mathcal{U}_3$). Moreover, we can verify that
 $F^\va_{\theta,\tau, \mu}\le \va^{\theta-\tau} G^\va_{\mu}$ for $\theta-\tau>0$. Thus   ($\mathcal{P}_1$)-($\mathcal{P}_2$) hold by Proposition \r{48} and \eq{wd} holds by Proposition \r{5b}.

Fix $R$ large enough and let
 $$
 \bar{w}_\mu=2\sup_{\va\in(0,\va_0)}\|v_\va\|_{L^\wq(\rn)}R^\mu w_\mu,\,\,\bar{v}_\va=v_\va-\bar{w}_\mu.
  $$
  Clearly, $\bar{v}_\va\le0$ in $B_R(0)$, $\bar{v}_\va\in \dot{H}^s(\rn)$ and $\int_{\rn}V_\va(\bar{v}_{\va,+})^2\le\int_{\rn}V_\va v^2_\va<\wq$. Moreover, from Proposition \r{prop4.1}, \eq{wd} and Proposition \r{cc5}, $\bar{v}_\va$  satisfies weakly
$$\Fs\bar{v}_\va+\frac{1}{2}V_{\va}\bar{v}_\va\le p\big(\I(\tilde{\P}\bar{v}_\va)\big)\tilde{\P}\ \ \mathrm{in}\ \R^N\backslash B_R(0).$$
It follows  from Proposition \r{4z} that
$\bar{v}_\va\le0$ in $\rn$. Then $v_\va\le Cw_\mu$. In particular, if $x\in \rn\setminus\La$, noting that $\liminf_{\va\to0}\dist(x_\va,\rn\setminus\La)>0$, it holds
\begin{align}\label{sjj}
u_\va(x)=v_\va\Big(\frac{x-x_\va}{\va}\Big)\le&C\Big(1+\Big|\frac{x-x_\va}{\va}\Big|^2\Big)^{-\frac{\mu}{2}}\nonumber\\
\le& \frac{C\va^\mu}{\va^\mu+|x-x_\va|^\mu}\le\frac{C\va^\mu}{|x|^\mu}.
\end{align}
This completes the proof.
\end{proof}

Finally, we prove Theorem $\ref{thm1.1}$.

\

\noindent\textbf{Proof of Theorem \ref{thm1.1}}:

{\bf  The case under the assumption ($\mathcal{Q}_1$), i.e. $p>1+\frac{\max\{s+\frac{\a}{2},\a\}}{N-2s}$.}

Let $\mu\in (\frac{N-2s}{2},N-2s)$  be sufficiently close to $N-2s$ from below, $\tau$ and $\theta$ be such that
 \begin{align}\label{lll}
 \max\left\{s+\frac{\a}{2},\a\right\}<\tau<\theta<\mu(p-1)<(N-2s)(p-1).
 \end{align}
 By \eq{lll} and Proposition \r{49}, ($\mathcal{P}_1$) and ($\mathcal{P}_2$) hold. Then we can find a nonnegative nontrivial weak solution $u_\va$ to \eq{eqs3.2} by Lemma \r{lem3.6}.  Moreover, by  \eq{lll} and \eq{wsq},
$$
u_\va^{p-1}\le\frac{C\va^{\mu(p-1)}}{|x|^{\mu(p-1)}}\le\frac{\va^\theta}{|x|^\tau}=\P\ \ \mathrm{in}\ \rn\setminus\La
$$
for $\va$ small enough. Hence $u_\va$ is indeed a solution to the original problem $\eq{eqs1.1}$.

Letting $\{x_{\varepsilon}\}_\va$ be given by Lemma \ref{jz}. $\eq{sjj}$ says
\begin{align}\label{sxs}
  u_\va\le\frac{C\va^\mu}{\va^\mu+|x-x_\va|^\mu}.
\end{align}
Moreover, by Lemmas \r{w} and \r{s}, we know that $u_\va\in L^\wq(\rn)\cap C_{\mathrm{loc}}^\si(\rn)$ for any $\si\in (0,\min\{2s,1\})$. It follows by Lemma \r{w61} that $u_\va>0$ in $\rn$.

Next, we derive a higher regular estimate of $u_\va$ if additionally  $V\in C_{\mathrm{loc}}^\varrho(\rn)\cap L^\wq(\rn)$ for some $\varrho\in (0,1)$.

Since $u_\va$ is a solution to $\eq{eqs1.1}$, we see that
$v_\va(y):=u_\va(x_\va+\va y)$ solves
\begin{align}\label{2b}
(-\Delta)^sv_\va = h_\va\ \ \text{in}\ \R^N,
\end{align}
where $h_\va(y)=-V(x_\va+\va y)v_\va+\big(\I (v_\va^p)\big)v_\va^{p-1}$.
 It suffices to prove that $\I v_\va^p\in C^\delta(\rn)$ for any $\delta\in (0,\min\{1,2s\})$.

In fact, if  $\I v_\va^p\in C^\delta(\rn)$,  it follows from Lemmas \ref{w}, \r{s} and the assumption $V\in C_{\mathrm{loc}}^\varrho(\R^N)$ that $h_\va\in C_{\mathrm{loc}}^\vartheta(\R^N)$ for some $\vartheta\in\left(0, \min\{1,2s,\varrho\}\right)$. Thus, for any given $R>1$, from  \c[Theorem 12.2.5]{clp}, we know  $v_\va\in C^{2s+\vartheta}(B_R(0))$ satisfying
$$\|v_\va\|_{C^{2s+\vartheta}(B_R(0))}\le C\Big(\|h_\va\|_{C^\vartheta(B_{3R}(0))}+\|v_\va\|_{L^\wq(\rn)}\Big).$$
Since $R>1$ is arbitrary, by rescaling, we deduce that $u_\va\in C_{\mathrm{loc}}^{2s+\vartheta}(\rn)$.

 In the following, we verify $\I v_\va^p\in C^\delta(\rn)$ for any $\delta\in (0,\min\{1,2s\})$. Actually, fix any $\delta\in (0,\min\{1,2s\})$,   from $\eq{g1}$, we have $v_\va\in C^\delta(\rn)$. By lemma \r{I}, we find $\I v_\va^p\in L^\wq(\rn)$. Besides, for any given $x_1,x_2\in \rn$, $x_1\neq x_2$, since $\mu(p-1)>\a$, we have
\begin{align}
&\frac{|\I (v_\va^p)(x_1)-\I (v_\va^p)(x_2)|}{|x_1-x_2|^\delta}\nonumber\\
\le&\int_{\rn}\frac{1}{|y|^{N-\a}}\frac{|v_\va^p(x_1-y)-v_\va^p(x_2-y)|}{|x_1-x_2|^\delta}\d y\nonumber\\
\le& C\|v_\va\|_{C^\delta(\rn)}\int_{\rn}\frac{1}{|y|^{N-\a}}\left(v_\va^{p-1}(x_1-y)+v_\va^{p-1}(x_2-y)\right)\d y\nonumber\\
\le&C\int_{\rn}\frac{1}{|x_1-y|^{N-\a}}\frac{1}{1+|y|^{\mu(p-1)}}\d y+C\int_{\rn}\frac{1}{|x_2-y|^{N-\a}}\frac{1}{1+|y|^{\mu(p-1)}}\d y\nonumber\\
\le&2C\sup_{x\in\rn}\int_{\rn}\frac{1}{|x-y|^{N-\a}}\frac{1}{1+|y|^{\mu(p-1)}}\d y\le C,\nonumber
\end{align}
where we use the fact that
  $v_\va^{p-1}(y)= u^{p-1}_\va(x_\va+\va y)\le\frac{C}{1+|y|^{\mu (p-1)}}$ by \eq{sxs}.

Therefore, $u_\va\in C_{\mathrm{loc}}^{2s+\vartheta}(\rn)$, and hence $u_\va$ is a classical solution to \eq{eqs1.1}.

The proofs for the other cases are similar, so we only give the corresponding choice of $p$ and parameters.

{\bf  The case under the assumption ($\mathcal{Q}_2$) with $\omega =2s$, i.e. $p>1+\frac{\a+2s}{N}$, $\inf_{x\in\rn}V(x)(1+|x|^{2s})>0$.}

Let $\mu\in (N-2s,N)$ be sufficiently close to $N$ from below, $\tau$ and $\theta$ satisfy
 \begin{align}
 \a+2s<\tau<\theta<\mu(p-1)<N(p-1).\nonumber
 \end{align}

{\bf   The case under the assumption ($\mathcal{Q}_2$) with  $\omega\in(0,2s)$, i.e.  $p>1+\frac{\a+2s}{N+2s-\omega}$, $\inf_{x\in\rn}V(1+|x|^{\omega})>0$ for  $\omega\in(0,2s)$.}

Let $\mu$ be sufficiently close to $N+2s-\omega$ from below, $\tau$ and $\theta$ satisfy
 \begin{align}
 \a+2s<\tau<\theta<\mu(p-1)<(N+2s-\omega)(p-1).\nonumber
 \end{align}
The proof of Theorem \r{thm1.1} is then completed.
\vspace{0.3cm}

Under specific decay assumptions on $V$, we can also get the lower decay estimates of $u_{\varepsilon}$. For example, taking $w_N:=\frac{1}{(1+|x|^2)^{N/2}}$, by \eq{wqb} and Proposition \r{tb}, we can verify that
\begin{align*}
  \va^{2s}(-\De)^sw_N+Vw_N\le -\frac{\va^{2s}C_N\ln|x|}{|x|^{N+2s}}+\frac{C}{1+|x|^{2s}}\frac{1}{(1+|x|^2)^{N/2}}<0,\quad |x|>R_\va,
\end{align*}
for some $R_\va>0$ large enough. On the other hand, letting $u_\va$ be a positive weak solution of \eq{eqs1.1}, it is clear that
\begin{align*}
   \va^{2s}(-\De)^su_\va+Vu_\va>0,\quad x\in\rn.
\end{align*}
It follows from comparison principle that
\begin{align*}
  w_N\le \frac{1}{\inf_{x\in B_{R_\va}(0)}u_\va}u_\va,\quad x\in \rn,
\end{align*}
i.e.,
\begin{align*}
  u_\va\ge  w_N\inf_{x\in B_{R_\va}(0)}u_\va\ge \frac{C_\va}{1+|x|^N}
\end{align*}
for some $C_\va>0$ since $u_\va>0$ in $\overline{B_{R_\va}(0)}$.
Thus we obtain the following remark:

\begin{remark}\label{rma}
Assume $p\in [2,\frac{N+\a}{N-2s})$, $p> 1+\frac{\a+2s}{N}$ and
\begin{equation}\label{wqb}
  c\le V(x)(1+|x|^{2s})\le C,\quad x\in\rn,
\end{equation}
for constants $C,c>0$. Let $u_\va$ be given by Theorem \r{thm1.1}. Then
\begin{align*}
  u_\va\ge \frac{C_\va}{1+|x|^N},
\end{align*}
for a constant $C_\va>0$ depending on $\va$.
\end{remark}

\vspace{0.3cm}

\section{Nonexistence results}\label{s6}
In this section, we aim to obtain some nonexistence results for \eq{eqs1.1}. Before that, we present the following comparison principle.
\begin{lemma}\label{ws8}(Comparison principle) Let  $f(x)\in L^1_{\mathrm{loc}}(\rn\backslash\{0\})$ with $f(x)\ge0$. Suppose $\tilde v\in \dot{H}(\rn)\cap C(\rn)$ with $\tilde v>0$  being  a weak supersolution to
\begin{align*}
  (-\De)^s v+Vv= f(x),\quad x\in \rn\backslash B_{R}(0),
\end{align*}
and $\underline{v}_\lambda \in \dot{H}(\rn)\cap C(\rn)$ with $\underline{v}_\lambda >0$ being  a weak subsolution to
\begin{align*}
  (-\De)^s {v}+V{v}= \la f(x),\quad x\in \rn\backslash B_{R'}(0),
\end{align*}
where $R,R',\la>0$ are constants. Then there holds
\begin{align*}
  \tilde v\ge C \underline {v}_\lambda,\quad x\in\rn,
\end{align*}
where $C>0$ is a constant depending only on $\la$, $\tilde{R}:=\max\{R,R'\}$, $\min_{B_{\tilde{R}}(0)} \tilde v$ and $\max_{B_{\tilde{R}}(0)}\underline{v}_\lambda$.
\end{lemma}
\begin{proof}
 Define
$$\bar{v}:=\min\Big\{1,\frac{\min_{B_{\tilde{R}}(0)}\tilde v}{\max_{B_{\tilde{R}}(0)}\underline{v}_\lambda}\Big\}\frac{1}{\max\{1,\la\}}\underline{v}_\lambda,\,\,w:=\bar{v}-\tilde v.
$$
 Clearly, $w\le0$ in $B_{\tilde{R}}(0)$ and $w$  weakly satisfies
\begin{align}\label{wq7}
  (-\De)^s w+Vw\le 0,\quad x\in \rn\backslash B_{\tilde{R}}(0).
\end{align}
Then by the same arguments as \eq{23j}, \eq{24j} and \eq{25j}, we get $w_+\le0$ in $\rn$, which  completes the proof.
\end{proof}

To prove Theorem \r{thm1.2'}, we need to give the following decay properties for the nonlocal Choquard term.
\begin{lemma}\label{l46}It holds that
\begin{equation}\label{qg8}
   I_\a*w^p_\mu\ge\frac{C}{|x|^{N-\a}}+\frac{C}{|x|^{\mu p-\a}},\quad |x|\ge2,
\end{equation}
where $C>0$ is a constant depending only on $N$, $\a$, $\mu$ and $p$.
\end{lemma}
\begin{proof}Let $|x|\ge2$.
\begin{align}\label{wh0}
 (I_\a*w^p_\mu)(x)\ge& \int_{B_{|x|/2}(x)}\frac{C}{|x-y|^{N-\a}|y|^{\mu p}}dy+\int_{B_{|x|/2}(0)}\frac{1}{|x-y|^{N-\a}(1+|x|^2)^{\frac{\mu p}{2}}}dy\nonumber\\
 &+\int_{B^c_{2|x|}(0)}\frac{C}{|x-y|^{N-\a}|y|^{\mu p}}dy\nonumber\\
 \ge&\frac{C}{|x|^{\mu p}}\int_{B_{|x|/2}(x)}\frac{C}{|x-y|^{N-\a}}dy+\frac{C}{|x|^{N-\a}}\int_{B_{|x|/2}(0)}\frac{1}{(1+|x|^2)^{\frac{\mu p}{2}}}dy\nonumber\\
 &+\int_{B^c_{2|x|}(0)}\frac{C}{|y|^{N-\a+\mu p}}dy\nonumber\\
 \ge&\frac{C}{|x|^{\mu p-\a}}+\frac{C}{|x|^{N-\a}}+\int_{B^c_{2|x|}(0)}\frac{C}{|y|^{N-\a+\mu p}}dy.
\end{align}
Note that
\begin{equation*}\int_{B^c_{2|x|}(0)}\frac{C}{|y|^{N-\a+\mu p}}dy=\left\{
  \begin{aligned}
    \frac{C}{|x|^{\mu p-\a}}&,\quad \mu p>\a,\\
    +\wq&,\quad \mu p\le \a,
   \end{aligned}\right.
\end{equation*}
The the conclusion follows immediately  by \eq{wh0}.
\end{proof}

Now we are going to prove Theorem \r{thm1.2'}.
 Without of loss generality, we may assume $\varepsilon =1$.  It suffices to consider the following equation
\begin{align}\label{wqs}
  (-\De)^su+V(x)u=(I_\a*|u|^{p})|u|^{p-2}u, \quad x\in\rn.
\end{align}

\begin{proof}[{\bf Proof of Theorem \r{thm1.2'}}]
Assume that $p\in (1,1+\frac{s+\frac{\a}{2}}{N-2s})\cup[2, 1+\frac{\a}{N-2s})$ and $\limsup_{|x|\to\wq}(1+|x|^{2s})V(x)=0$. Then for given $\epsilon>0$, $V(x)\le \frac{\epsilon}{1+|x|^{2s}}$ in $\rn\backslash B_{R_\epsilon}(0)$ for some $R_\epsilon>0$. Afterwards,  $\epsilon>0$ can be taken smaller if necessary.

Suppose by contradiction that $u\in H^s_{V,1}(\rn)\cap C(\rn)$ is a nonnegative nontrivial weak solution to \eq{wqs}. There holds
\begin{align}\label{wq81}
  \int_{\rn}(I_\a*u^p)u^p=[u]_s^2+\int_{\rn}Vu^2<\wq.
\end{align}
Moreover, by Lemma \r{w61}, $u>0$ in $\rn$.

Let $\mu_1\in (N-2s,N)$ be a parameter.  By Propositions \r{tb} and \r{cc5}, $w_{\mu_1}$ weakly satisfies
\begin{align}\label{qy9}
  (-\De)^s w_{\mu_1}+V(x)w_{\mu_1}\le-\frac{C_{\mu_1}}{|x|^{\mu_1+2s}}+\frac{\epsilon}{|x|^{\mu_1+2s}}\le0,\quad x\in \rn\backslash B_{R_1}(0)
\end{align}
for some $R_1>0$.
It follows by \eq{qy9} and Lemma \r{ws8} that
\begin{align}\label{24f}
  u\ge C_1w_{\mu_1}
\end{align}
for a constant $C_1>0$.

Now we divide the proof  into the following two cases.

{\bf Case 1: $1< p<1+\frac{s+\frac{\a}{2}}{N-2s}$}.

By Lemma \r{l46}, we have
\begin{align}\label{pq9}
  I_\a*w_\mu^p\ge \frac{C}{|x|^{\mu p-\a}},\quad |x|\ge2.
\end{align}
Choose  $\mu_2\in (\frac{N-2s}{2},N-2s)$ and $\mu_1\in (N-2s,N)$ such that
\begin{equation}\label{w60}
  N>\mu_2+2s>\mu_1(2p-1)-\a.
\end{equation}
From \eq{wqs}, \eq{24f} and \eq{pq9}, we get
\begin{align*}
  (-\De)^su+Vu\ge\frac{C}{|x|^{\mu_1(2p-1)-\a}},\quad |x|\ge2.
\end{align*}
In addition, Proposition \r{tb}, Proposition \r{cc5} and \eq{w60} indicate that $w_{\mu_2}$ weakly satisfies
\begin{align}\label{ro9}
  (-\De)^s w_{\mu_2}+V(x)w_{\mu_2}\le \frac{C_{\mu_2}}{|x|^{\mu_2+2s}}+\frac{\epsilon}{|x|^{\mu_2+2s}}\le \frac{C}{|x|^{\mu_1(2p-1)-\a}},\quad x\in \rn\backslash B_{R_2}(0)
\end{align}
for some $R_2>0$.
As a consequence of Lemma \r{ws8}, there exists $C_2>0$ such that
\begin{align*}
  u\ge C_2w_{\mu_2}.
\end{align*}
It follows from \eq{pq9} that
\begin{align*}
  (-\De)^su+Vu\ge\frac{C}{|x|^{\mu_2(2p-1)-\a}},\quad |x|\ge2.
\end{align*}
Set $\mu_{i+1}:=\mu_i(2p-1)-\a-2s$, $i\ge 2$, i.e.,
$$\mu_{i}=(2p-1)^{i-2}\Big(\mu_{2}-\frac{\a+2s}{2p-2}\Big)+\frac{\a+2s}{2p-2},\quad i\ge2.$$
Due to $2p-1>1$ and $\mu_2<N-2s<\frac{\a+2s}{2p-2}$, it follows that $\mu_{i+1}<\mu_{i}<N-2s$ for $i\ge2$ and $\mu_i\to-\wq$ as $i\to\wq$.

Fix $i\ge2$ such that $\mu_i>\frac{N-2s}{2}, \mu_{i+1}>\frac{N-2s}{2}$. We claim that there exists constants $C_{i},C_{i+1}>0$ such that
\begin{align}\label{d0n}
  u\ge C_{i+1}w_{\mu_{i+1}}\ \mathrm{if}\ u\ge C_{i}w_{\mu_{i}}.
\end{align}
In fact, if $u\ge C_iw_{\mu_i}$,   then by  \eq{pq9},
\begin{align*}
  (-\De)^s u+V(x)u\ge \frac{C}{|x|^{\mu_{i}(2p-1)-\a}},\quad x\in \rn\backslash B_1(0).
\end{align*}
On the other hand, thanks to Proposition \r{tb} and Proposition \r{cc5}, $w_{\mu_{i+1}}$ weakly satisfies
\begin{align*}
  (-\De)^s w_{\mu_{i+1}}+V(x)w_{\mu_{i+1}}\le & \frac{C_{\mu_{i+1}}}{|x|^{\mu_{i+1}+2s}}+\frac{\epsilon}{|x|^{\mu_{i+1}+2s}}\\
  \le & \frac{C}{|x|^{\mu_i(2p-1)-\a}},\quad x\in \rn\backslash B_{R_i}(0)
\end{align*}
for some $R_i>0$. As a consequence of Lemma \r{ws8}, the claim \eq{d0n} holds immediately.

Therefore, for any $\mu>\frac{N-2s}{2}$, by finite iteration from \eq{d0n}, we obtain
\begin{align*}
  u\ge d_\mu w_{\mu},\quad x\in\rn
\end{align*}
for some constant $d_\mu>0$.
Choosing $\mu>\frac{N-2s}{2}$ such that $2\mu p-\a<N$, we get
\begin{align*}
  \int_{\rn}(I_\a*u^p)u^p\ge C\int_{\rn\backslash B_1(0)}\frac{1}{|x|^{2\mu p-\a}}=+\wq,
\end{align*}
which contradicts to \eq{wq81}.

{\bf Case 2:  $2\le p<1+\frac{\a}{N-2s}$}.

Reviewing Lemma \r{l46}, in this case,  we will apply the following estimate instead of \eq{pq9} in Case 1,
\begin{align}\label{pq9'}
  I_\a*w_\mu^p\ge \frac{C}{|x|^{N-\a}},\quad |x|\ge2.
\end{align}

Since $2<1+\frac{\a}{N-2s}$, we have $\a>N-2s$.
Pick  $\mu_2\in (\frac{N-2s}{2},N-2s)$ and $\mu_1\in (N-2s,N)$  such that
\begin{equation}\label{w6o}
  N>\mu_2+2s>\mu_1(p-1)+N-\a.
\end{equation}
Through \eq{wqs}, \eq{24f} and \eq{pq9'}, we get
\begin{align*}
  (-\De)^su+Vu\ge\frac{C}{|x|^{N-\a+\mu_1(p-1)}},\quad |x|\ge2.
\end{align*}
On the other hand, Proposition \r{tb}, Proposition \r{cc5} and \eq{w6o} imply that $w_{\mu_2}$ weakly satisfies
\begin{align*}
  (-\De)^s w_{\mu_2}+V(x)w_{\mu_2}&\le \frac{C_{\mu_2}}{|x|^{\mu_2+2s}}+\frac{\epsilon}{|x|^{\mu_2+2s}}\\
  &\le \frac{C}{|x|^{N-\a+\mu_1(p-1)}},\quad x\in \rn\backslash B_{R_2}(0)
\end{align*}
for some $R_2>0$. Hence,  by Lemma \r{ws8}, there exists $C_2>0$ such that
\begin{align*}
  u\ge C_2w_{\mu_2}.
\end{align*}
 It follows from \eq{pq9'} that
\begin{align*}
  (-\De)^su+Vu\ge\frac{C}{|x|^{N-\a+\mu_2(p-1)}},\quad |x|\ge2.
\end{align*}
Set $\mu_{i+1}:=\mu_i(p-1)+N-\a-2s$, $i\ge 2$, i.e.,
\begin{equation*}
  \begin{aligned}
     &\mu_{i}=\mu_2+(i-2)(N-\a-2s),\quad i\ge2,\quad&\ \mathrm{if}\ p=2;\\
     &\mu_{i}=(p-1)^{i-2}\Big(\mu_2+\frac{N-\a-2s}{p-2}\Big)+\frac{N-\a-2s}{2-p},\quad i\ge2,\quad&\ \mathrm{if}\ p>2.
   \end{aligned}
\end{equation*}
Since $\a>N-2s$ and $\mu_2+\frac{N-\a-2s}{p-2}<N-2s+\frac{N-\a-2s}{p-2}<0$ for $p<1+\frac{\a}{N-2s}$,
it follows that $\mu_{i}<N-2s$ for $i\ge2$ and $\mu_i\to-\wq$ as $i\to\wq$.

By finite iterations similar to those in Case 1, for any $\mu>\frac{N-2s}{2}$, we can find a constant $d_\mu>0$ satisfying
\begin{align*}
  u\ge d_\mu w_{\mu},\quad x\in\rn.
\end{align*}
Setting $\mu>\frac{N-2s}{2}$ such that $\mu p+N-\a<N$, we derive
\begin{align*}
  \int_{\rn}(I_\a*u^p)u^p\ge C\int_{\rn\backslash B_1(0)}\frac{1}{|x|^{\mu p+N-\a}}=+\wq,
\end{align*}
which contradicts to \eq{wq81}.

As a result, we complete the proof of Theorem \r{thm1.2'}.
\end{proof}


\end{document}